\documentclass[a4paper,11pt]{amsart}

\usepackage{amssymb,amsmath,mathrsfs,url,float,hyperref}   
\usepackage[numbers, sort&compress]{natbib}
\usepackage[inner=3cm,outer=3cm,top=3cm,bottom=3cm,includehead,includefoot]{geometry}
\hypersetup{citecolor=red, linkcolor=blue, colorlinks=true}
\newtheorem{thm}{Theorem}[section]
\newtheorem{prop} [thm]{Proposition}
\newtheorem{cor} [thm]{Corollary}
 \newtheorem{lemma} [thm]{Lemma}

\theoremstyle{definition}
\newtheorem{remark}[thm]{Remark}

\renewcommand\leq{\leqslant} 
\renewcommand\geq{\geqslant}

\newcommand{\Deg}[1]{|#1|}

\DeclareMathOperator{\End}{End}
\DeclareMathOperator{\Char}{char}
\DeclareMathOperator{\sgn}{sgn}
\DeclareMathOperator{\Span}{span}
\DeclareMathOperator{\GL}{GL}

\title{Regular orbits of symmetric and alternating groups}

\author{Joanna B. Fawcett}
\address[J.B.~Fawcett]{ Centre for the Mathematics of Symmetry and Computation, School of Mathematics and Statistics, The University of Western Australia, 35 Stirling Highway, Crawley, WA 6009, Australia.}
\email{joanna.fawcett@uwa.edu.au}


\author{E. A. O'Brien}
\address[E.A.~O'Brien]{Department of Mathematics, University of Auckland, Private Bag 92019, Auckland, New Zealand}
\email{obrien@math.auckland.ac.nz}

\author{Jan Saxl}
\address[Jan~Saxl]{Department of Pure Mathematics and Mathematical Statistics, Centre for Mathematical Sciences, University of Cambridge, Wilberforce Road, Cambridge CB3 0WB, United Kingdom}
\email{j.saxl@dpmms.cam.ac.uk
}

\thanks{The first author was supported by the Australian Research Council  Discovery Project grant DP130100106, and,  while a Ph.D. student at the University of Cambridge, by the Cambridge Commonwealth Trust  and St John's College, Cambridge. All authors were supported in part by the Marsden Fund of New Zealand via 
grant UOA 105. We thank the referee for helpful 
comments and J\"urgen M\"uller for providing us with a draft of \cite{Mul2011}. Published: \textit{J. Algebra 458 (2016), 21--52. DOI: 10.1016/j.jalgebra.2016.02.018.}
}

\begin{document}
\maketitle

\begin{abstract}
 Given a finite group $G$ and a faithful irreducible $FG$-module $V$ where $F$ has prime order, does $G$ have a regular orbit on  $V$? This problem is equivalent to determining which primitive permutation groups of affine type have a base of size 2. In this paper, we  classify the pairs $(G,V)$ for which $G$ has a regular orbit on $V$  where $G$ is a covering group of a symmetric or alternating group and $V$ is a faithful irreducible $FG$-module such that the order of $F$ is prime and divides the order of $G$.
 \end{abstract}

\section{Introduction}

Let $G$ be a finite group acting faithfully on a set $\Omega$. A \textit{base} $\mathscr{B}$ for $G$ is a non-empty subset of $\Omega$ with the property that only the identity fixes every element of $\mathscr{B}$; if   $\mathscr{B}=\{\omega\}$  for some $\omega\in \Omega$, then the orbit $\{\omega g : g\in G\}$ of $G$ on $\Omega$ is \textit{regular}. Bases have been very useful in permutation group theory in the past half century, both theoretically in bounding the order of a primitive permutation group in terms of its degree  (e.g.,\  \cite{Bab1981}) and  computationally (cf.\ \cite{Ser2003}). Recently, much work has been done on classifying the finite primitive permutation groups of almost simple and diagonal type  with a base of size 2 \cite{BurGurSax2011,BurGurSaxS,BurObrWil2010,Faw2013}. In this paper, we consider this problem for primitive permutation groups of affine type.

A finite permutation group $X$ is \textit{affine} if its socle  is a finite-dimensional $\mathbb{F}_p$-vector space $V$ for some  prime $p$, in which case $X= V:X_0$ and $X_0\leq\GL(V)$, where $X_0$ denotes the stabiliser of the vector $0$ in $X$. Such a group $X$ is primitive precisely when $V$ is an irreducible $\mathbb{F}_pX_0$-module, in which case we say that $X$ is a \textit{primitive permutation group of affine type}. Note that $X$ has a base of size $2$ on $V$ if and only if  $X_0$ has a regular orbit on  $V$. Thus classifying the primitive permutation groups of affine type with a  base of size 2 amounts to determining which finite   groups $G$, primes $p$, and faithful irreducible  $\mathbb{F}_pG$-modules $V$  are such that $G$ has a regular orbit on $V$.

More generally, given a finite group $G$ and a faithful $FG$-module $V$ where $F$ is any field, we can ask whether $G$ has a regular orbit on $V$.  This problem is of independent interest to   representation theorists. Indeed, the classification  of the pairs $(G,V)$  for which $G$ has no regular orbits on $V$ where $G$ is a $p'$-group that normalises a quasisimple group  acting irreducibly on the faithful $\mathbb{F}_pG$-module $V$ \cite{Lie1996,Goo2000,KohPah2001}    
provided an important contribution to the solution of the famous $k(GV)$-problem \cite{SchP2007}, which proved part of a well-known conjecture  of Brauer  concerning defect groups of blocks \cite{Bra1956}.

However, little work has been done on the  regular orbit problem in the case where the characteristic of the field divides the order of the group. Hall, Liebeck and Seitz \cite[Theorem 6]{HalLieSei1992} proved   that  if $G$ is a finite quasisimple group with no regular orbits on a faithful irreducible $FG$-module $V$ where $F$ is a field of characteristic $p$, then either $G$ is of Lie type in characteristic $p$, or $G= A_n$ where $p\leq n$ and $V$ is the fully deleted permutation module (cf.\ \S \ref{s: Rn(1)}), or $(G,V)$ is one of finitely many exceptional pairs. These exceptional pairs are not known in general. Motivated by this result, we classify the pairs $(G,V)$ for which $G$ has a regular orbit on $V$ where $G$ is a   scalar extension of a covering group $H$ of the symmetric group $S_n$ or the alternating group $A_n$ and $V$ is a faithful irreducible $\mathbb{F}_pH$-module such that $p\leq n$. The case where $p>n$ follows from \cite{Goo2000,KohPah2001}.

Let $S$ be a finite group. A finite group $L$ is a   \textit{covering group} or  \textit{cover} of  $S$  if  $L/Z(L)\simeq S$ and $Z(L)\leq L'$. We say that $L$ is a \textit{proper} covering group when $Z(L)\neq 1$. The proper covering groups  of  $S_n$ for $n\geq 5$  are $2.S_n^+$ and $2.S_n^-$, and these groups are isomorphic precisely when $n=6$ \cite{Sch1911,HofHum1992}. The proper covering groups of  $A_n$ are   $2.A_n$ for $n\geq 5$, and $3.A_n$ and $6.A_n$ for $n=6$ or $7$ \cite{Sch1911,HofHum1992}.  The following is our main result.

\begin{thm}
\label{regular total}
Let $H$ be a covering group of  $S_n$ or $A_n$ where $n\geq 5$. Let $G$ be a group for which  $H\leq G\leq H\circ \mathbb{F}_p^*$ where $p$ is a prime and $p\leq n$. Let $V$ be a faithful irreducible $\mathbb{F}_pH$-module, and let $d:=\dim_{\mathbb{F}_p}(V)$. 
\begin{itemize}
\item[(i)] If either $V$ or $V\otimes_{\mathbb{F}_p} \sgn$ is the fully deleted permutation module of $S_n$, then $G$ has a regular orbit on $V$ if and only if $G=A_n$ and $p=n-1$.
\item[(ii)]  If neither $V$ nor $V\otimes_{\mathbb{F}_p} \sgn$ is the fully deleted permutation module of $S_n$, then  $G$ has a regular orbit on $V$ if and only if  $(n,p,G,d)$ is not    listed in Table $\ref{tab: total ex}$.
\end{itemize}

\end{thm}

\begin{table}[!h]
\renewcommand{\baselinestretch}{1.1}\selectfont
\centering
\begin{tabular}{ l l l l  }
\hline
$n$ & $p$  & $G$ &   $d$    \\
\hline 
  5 & 2   &   $A_5$, $S_5$ & 4   \\
   & 3  & $A_5\times \mathbb{F}_3^*$, $S_5\times \mathbb{F}_3^*$ & 6  \\
   &  & $2.A_5$, $2.S_{5}^+$, $2.S_{5}^-$ & 4 \\

 & 5 &  $H=2.A_5$ & 2  \\
 & & $2.S_5^-\circ\mathbb{F}_5^*$,  $H=2.S_{5}^+$   & 4  \\
 & & $2.A_5\circ\mathbb{F}_5^*$  & 4  \\

  6 & 2  &  $A_6$, $S_6$  & 4   \\
  
  & & $3.A_6$ & 6 \\

  & 3 &   $H\in\{A_6,S_6\}$  & 6  \\
  
  &  &  $2.A_6$, $2.S_{6}$ & 4 \\

  & 5& $H\in\{A_6,S_6\}$, $G\neq A_6$   & 5    \\

 &   &  $H=2.A_6$ & 4   \\

& & $G\neq H=3.A_6$  & 6 \\

 7 & 2   & $A_7$ & 4   \\

  &  &  $S_7$  & 8, 14  \\
  & & $3.A_7$ & 12 \\
  
   & 3 &  $2.A_7$, $2.S_{7}^+$, $2.S_{7}^-$ & 8  \\
   \hline
\end{tabular}
\quad
\begin{tabular}{ l l l l  }
\hline
$n$ & $p$  & $G$ &   $d$    \\
\hline 
 7  & 5 & $H=3.A_7$ & 6 \\
 
 & 7  &  $H\in\{2.A_7,2.S_7^-\}$ & 4 \\
 & & $H=3.A_7$ & 6 \\

8 & 2   &  $A_8$ & 4, 14   \\ 

& &  $S_8$ & 8, 14    \\

 & 3 &  $2.A_{8}$, $2.S_{8}^-$ & 8 \\

 & 5 & $2.S_8^+\circ\mathbb{F}_5^*$, $H=2.S_{8}^-$ & 8  \\

 9 & 2   & $A_9$ & 8, 20   \\
& &   $S_9$ & 16  \\

 & 3 &  $2.A_{9}$, $2.S_{9}^-$ & 8  \\

 & 5 &  $H=2.A_{9}$ & 8  \\

10 & 2  &  $A_{10}$, $S_{10}$ & 16  \\

 & 3 &  $2.A_{10}$, $2.S_{10}^-$ & 16  \\

 & 5 & $H=2.A_{10}$ & 8  \\

11 & 3 & $2.A_{11}$ & 16 \\
12 & 2   & $S_{12}$ & 32     \\
& 3 & $2.A_{12}$ & 16 \\
\hline
\end{tabular}
\caption{$\mathbb{F}_pG$-modules $V$ on which $G$ has no regular orbits}
\label{tab: total ex}
\renewcommand{\baselinestretch}{1.3}\selectfont
\end{table}

The fully deleted permutation module is a faithful irreducible $\mathbb{F}_pS_n$-module of dimension $n-1$ when $p\nmid n$ and dimension $n-2$ otherwise; its restriction to $A_n$ is always irreducible (cf.\ \S \ref{s: Rn(1)}). The definitions of $\sgn$ and $H\circ \mathbb{F}_p^*$ are given in Section \ref{s: prelim}.

When  $H$ is specified in  Table \ref{tab: total ex}, we mean that  $H\leq G\leq H\circ\mathbb{F}_p^*$ with no restrictions on $G$. Also, for certain $d$ listed in Table \ref{tab: total ex}, there  exist multiple  faithful irreducible $\mathbb{F}_pH$-modules of dimension $d$, none of which admit regular orbits; this includes the case where $A_n\leq H$ and $d$ is  the dimension of the fully deleted permutation module.

Theorem \ref{regular total} follows  from Theorem \ref{regular} and Remark \ref{rem: dim}  in the case where $H$ is $S_n$ or $A_n$, and Theorem \ref{regular double} in the case where $H$ is a proper covering group of $S_n$ or $A_n$.  
Moreover, Theorems \ref{regular} and \ref{regular double} are consequences of more general results concerning regular orbits of central extensions of almost simple groups with socle $A_n$ (cf.\ Lemmas \ref{An reduction} and \ref{2An reduction}).

Observe that when $A_n\leq H$ and $n\geq 7$, representations only occur in Table \ref{tab: total ex}  for $p=2$. Moreover, when $2.A_n\leq H$, every representation listed in Table \ref{tab: total ex} is a  basic spin module (cf.\ \S \ref{s: cover props})  except when $(n,p,G,d)= (5,5,2.A_5\circ\mathbb{F}_5^*,4)$. 

It is well known that the group algebras of $2.S_n^+$ and $2.S_n^-$ are isomorphic over every field containing a primitive fourth root of unity. 
Thus, over such  fields, the representation theory of $2.S_n^+$ and $2.S_n^-$ is essentially the same, and typically, in order to answer a representation theoretical question, it suffices to consider one of the double covers. However, this is not the case for the regular orbit problem. Indeed, even over a splitting field  containing a primitive fourth root of unity, there is an example where only one double cover has a regular orbit; this occurs for $(n,p)=(8,5)$ in Table \ref{tab: total ex}.
Other examples occur for $(n,p)=(5,5)$, in which case $\mathbb{F}_p$ is not a splitting field but  contains a primitive fourth root of unity, and $(n,p)=(7,7)$, $(8,3)$, $(9,3)$ or $(10,3)$, in which case  $\mathbb{F}_p$ does not contain a primitive fourth root of unity.

As an immediate consequence of Theorem \ref{regular total}, we  obtain a result concerning bases  of primitive permutation groups of affine type.

\begin{cor}
Let $X$ be a primitive permutation group of affine type with socle $V\simeq \mathbb{F}_p^d$ where $p$ is a prime. Suppose that  $H\leq X_0\leq H\circ \mathbb{F}_p^*$ where $H$ is a covering group of $S_n$ or $A_n$ for $n\geq 5$, and assume that $p\leq n$.
\begin{itemize}
\item[(i)] If either $V$ or $V\otimes_{\mathbb{F}_p} \sgn$ is the fully deleted permutation module of $S_n$, then $X$ has a base of size $2$ on $V$ if and only if $X_0=A_n$ and $p=n-1$.
\item[(ii)] If neither $V$ nor $V\otimes_{\mathbb{F}_p} \sgn$ is the fully deleted permutation module of $S_n$, then  $X$ has a base of size $2$ on $V$ if and only if  $(n,p,G,d)$ is not    listed in Table $\ref{tab: total ex}$ where $G:=X_0$. 
\end{itemize}
\end{cor}

In fact, it can be established by routine computations  using {\sc 
Magma} \cite{Magma} that  for  $(n,p,X_0,d)$  listed in Table~\ref{tab: total ex} with $n\geq 7$, the affine group $X$ has a base of size $3$ with the following exceptions: 
$(7,2,A_7,4)$, 
$(8,2,A_8,4)$, $(8,2,S_8,8)$ and $(9,2,A_9,8)$, in which case $X$ has a  base of minimal size $4$, $5$, $4$ and $4$ respectively. When (i) holds, the minimal base size  of $X$ cannot be constant in general, for $|X|$ is not bounded above by $|V|^c$ for any absolute constant $c$. In either case, $d+1$ is an upper bound on the minimal base size, for any basis of $V$ is a base for $X_0$.

This paper is organised as follows. In \S \ref{s: prelim} we collect some notation, definitions and basic results, and in \S \ref{s: bounds} we determine some  bounds for the dimensions of faithful representations   admitting no regular orbits.  In \S \ref{s: irred Sn} we consider the regular orbits of $S_n$ and $A_n$, and in \S \ref{s: cover props}   the regular orbits of the proper covering groups of $S_n$ and $A_n$. In \S \ref{s: comp} we briefly comment on our computational methods.

\section{Preliminaries}
\label{s: prelim}

 Unless otherwise specified, all groups in this paper  are finite, and all homomorphisms and actions are written on the right.  

Let  $G$ be a finite group. We denote the derived subgroup of $G$ by $G'$, the centre of $G$ by $Z(G)$, the conjugacy class of $g\in G$ by $g^G$, and the generalised Fitting subgroup of $G$ by $F^*(G)$  (cf.\ \cite{Asc2000} for a definition). The group $G$ is \textit{quasisimple} if $G=G'$ and $G/Z(G)$ is simple, and \textit{almost quasisimple} if $G/Z(G)$ is almost simple.

\begin{lemma}
\label{quasi}
Let $G$ be a finite almost quasisimple group. 
\begin{itemize}
\item[(i)] $F^*(G)=F^*(G)'Z(G)$ and $F^*(G)/Z(G)$ is the socle of $G/Z(G)$.
\item[(ii)] $F^*(G)'$ is quasisimple and $Z(F^*(G)')=F^*(G)'\cap Z(G)$.
\end{itemize}
\end{lemma}

\begin{proof}
Follows from \cite[Lemma 2.1]{GurTie2005} and \cite[31.1]{Asc2000}.
\end{proof}

In particular, if $G$ is an almost quasisimple group and the socle of $G/Z(G)$ is $A_n$, then $F^*(G)'$ is a quasisimple group with $F^*(G)'/Z(F^*(G)')\simeq A_n$. Hence $F^*(G)'$ is a covering group of $A_n$, so $F^*(G)'$ is one of $A_n$ or   $2.A_n$ for $n\geq 5$, or $3.A_n$ or $6.A_n$ for $n=6$ or $7$. We will consider the regular orbit problem for almost quasisimple groups $G$ with $F^*(G)'=A_n$  in \S \ref{s: irred Sn} (cf.\ Lemma \ref{An reduction}) and  $F^*(G)'=2.A_n$ in \S \ref{s: cover props} (cf.\ Lemma \ref{2An reduction}).

Let $F$ be a field. We denote the  characteristic of $F$ by $\Char(F)$, the multiplicative group of $F$ by $F^*$, and the group algebra of  $G$ over $F$ by $FG$. All $FG$-modules in this paper are finite-dimensional, and we denote the dimension of an $FG$-module $V$ by $\dim_F(V)$.
We denote the finite field of order $q$ by $\mathbb{F}_q$. 

Let $V$ be an $FG$-module. We say that  $V$ can be \textit{realised over} a subfield $K$ of $F$ if there exists an $F$-basis $\mathscr{B}$ of $V$ such that the matrix of the $F$-endomorphism $g$ of $V$ relative to $\mathscr{B}$ has entries in $K$ for every $g\in G$. If $F$ is a finite field and $V$ has character $\chi$, then $V$ can be realised over $K$ if and only if $K$ contains $\chi(g)$ for all $g\in G$ \cite[Theorem VII.1.17]{BlaHup1981}.

 For an extension field $E$ of $F$ and an $FG$-module $V$, 
we denote the extension of scalars of $V$ to $E$ by  $V\otimes_F E$ (cf.\ \cite{CurRei1962} for a definition). An irreducible $FG$-module $V$ is \textit{absolutely irreducible} if $V\otimes_F E$ is irreducible for every field extension $E$ of $F$. Note that $V$ is absolutely irreducible if and only if  $\End_{FG}(V)=F$ \cite[Theorem 29.13]{CurRei1962}, where $\End_{FG}(V)$ denotes the set of $FG$-endomorphisms of $V$. The field $F$ is a \textit{splitting field} for $G$ if every irreducible $FG$-module is absolutely irreducible. 

Let $F$ be a finite field, $H$ a finite group and  $V$ a faithful $FH$-module, and let $S(H)$ be the set of $h\in H$ for which there exists $\lambda_h\in F^*$ such that $vh=\lambda_h v$ for all $v\in V$. Note that $S(H)\leq Z(H)$. The \textit{central product}  of $H$ and $F^*$, denoted by $H\circ F^*$, is  the quotient $(H\times F^*)/N$ where $N=\{(h,\lambda_h^{-1}):h\in S(H)\}$. The $FH$-module $V$ naturally becomes a faithful $F(H\circ F^*)$-module under the action $vN(h,\lambda):=(\lambda v)h$ for all $v\in V$, $h\in H$ and $\lambda\in F^*$. Now $V$ is an irreducible $F(H\circ F^*)$-module if and only if $V$ is an irreducible $FH$-module. 
Moreover, if $V$ has dimension $d$ and $\rho$ is the corresponding representation of $H$ in $\GL_d(F)$, then $H\circ F^*\simeq \langle H\rho,F^*\rangle=H\rho F^*$.
If $V$ is a faithful irreducible $FH$-module and $|Z(H)|\leq 2$, then $S(H)=Z(H)$,  for a central involution must act as $-1$ on $V$.

\begin{lemma}
 \label{scalar}
 Let $F$ be a field, $G$ a finite group and $V$ an absolutely irreducible $FG$-module. For each $g\in Z(G)$, there exists $\lambda_g\in F^*$ such that $vg=\lambda_gv$ for all $v\in V$. If $V$ is faithful, then the map $g\mapsto \lambda_g$ for all $g\in Z(G)$ is an injective homomorphism from $Z(G)$ to $F^*$.\end{lemma}

\begin{proof}
Let $g\in Z(G)$. The $F$-endomorphism of $V$ defined by $v\mapsto vg$ for all $v\in V$ lies in $\End_{FG}(V)=F$, so the first claim holds. The second is straightforward.
\end{proof}

When $F$ is a finite field and $V$ is a (faithful) irreducible $FG$-module, we can use the field $k:=\End_{FG}(V)$ to construct a (faithful) absolutely irreducible representation of $G$ with the same $G$-orbits as $V$. Define $k$-scalar multiplication on the additive group $V$ to be evaluation, and let $G$ act in the same way. Now $V$ is a (faithful) absolutely irreducible $kG$-module since $\End_{kG}(V)\subseteq \End_{FG}(V)=k$, and clearly $G$ has a regular orbit on the $FG$-module $V$ if and only if $G$ has a regular orbit on the $kG$-module $V$.

 Let $H$ be a subgroup of $G$, and let $V$ be an $FG$-module. We denote the restricted module of $V$ from $G$ to $H$ by $V\downarrow H$.  We say that $V\downarrow H$ \textit{splits} if it is not irreducible.

\begin{lemma}
\label{H to G}
 Let $G$ be a finite group with subgroup $H$, and let $F$ be a field. If $W$ is an irreducible $FH$-module, then there is an irreducible $FG$-module $V$ for which $W\leq V\downarrow H$.
\end{lemma}

\begin{proof}
Follows from Frobenius-Nakayama reciprocity  \cite[Theorem VII.4.5]{BlaHup1981}. 
\end{proof}

Let $N$ be a normal subgroup of $G$, $V$ an  $FG$-module and $W$ an irreducible $FN$-submodule of $V$. For $g\in G$, the normality of $N$ implies that the \textit{conjugate} $Wg$ is  an irreducible $FN$-submodule of $V$. The following is well known from Clifford theory.

\begin{lemma}
\label{index 2} 
Let $G$ be a finite group, $N\unlhd G$ and $F$ a field. Let $V$ be an irreducible $FG$-module and $W$ an irreducible $FN$-submodule of $V$. Then $V\downarrow N$ is a direct sum of conjugates of $W$, and if $[G:N]=2$, then $V\downarrow N=W$ or $W\oplus Wg$ for all $g\in G\setminus N$. 
\end{lemma}

\begin{proof}
Since $\sum_{g\in G}Wg$ is an $FG$-submodule of $V$, it is equal to $V$, and so the first claim holds. If $[G:N]=2$ and $g\in G\setminus N$, then $V=W+ Wg$, so the second claim holds.
\end{proof}

Let $N$ be an index 2 subgroup of $G$. The \textit{sign module}, denoted by $\sgn$, is the one-dimensional $FG$-module for which $g\in N$ acts as $1$ and $g\in G\setminus N$ acts as $-1$. For an  $FG$-module $V$, the \textit{associate} of $V$ is the $FG$-module $V\otimes_F \sgn$ where $(v\otimes \lambda)g:=(vg)\otimes (\lambda g)$ for all $v\in V$, $\lambda\in \sgn$ and $g\in G$.

\section{Bounds for dimensions of non-regular representations}
\label{s: bounds}

In this section, we determine bounds for the dimensions of faithful irreducible representations of almost quasisimple groups that  admit no regular orbits. These  are obtained  using the standard technique  of counting fixed points.

 Let $G$ be a finite  group,   $F$  a field, and  $V$  an $FG$-module. We define   $C_V(g):=\{v\in V:vg=v\}$ for all $g\in G$.    For $X\subseteq G$, we define $[V,X]:=\Span\{v-vg:v\in V,g\in X\}$, and when $X=\{g\}$, we write $[V,g]$. Note that  $\dim_F(V)=\dim_F(C_V(g))+\dim_F([V,g])$ for all $g\in G$. For $v\in V$,  we denote the stabiliser of $v$ in $G$ by $C_G(v)$.

The following is a simple but crucial result.

\begin{lemma}
\label{dag}
 Let $G$ be a finite group and $F$  a field. Let $V$ be a faithful $FG$-module. If $G$ has no regular orbits on $V$, then
$V=\bigcup_{ g\in G\setminus\{ 1\}} C_V(g).$
\end{lemma}

\begin{proof}
If  $v\in V$ and $v\notin C_V(g)$ for all $1\neq g\in G$, then $v$ lies in  a regular orbit of $G$.
\end{proof}

 Lemma \ref{dag} implies that if $V$ is a faithful $FG$-module where $G$ is finite and $F$ is infinite, then $G$ has a regular orbit on  $V$, for no vector space over an infinite field is a finite union of proper subspaces.
Moreover, Lemma \ref{dag}  gives us a  bound for the size of $V$  that is easily computed using {\sc Magma}. To see this, we need the following useful observation about fixed points of central elements.

\begin{lemma}
\label{center dag} 
Let $G$ be a finite group and $F$  a  field. Let $V$ be a faithful irreducible $FG$-module. If $1\neq g\in  Z(G)$, then $C_V(g)=0$.
\end{lemma}

\begin{proof}
This follows from the fact that $C_V(g)$ is a proper $FG$-submodule of $V$.
\end{proof}

Now we provide the bound mentioned above. 

\begin{lemma}
\label{strong bound}
Let $G$ be a finite group and $F$  a finite field. Let $V$ be a faithful irreducible $FG$-module. If $G$ has no regular orbits on  $V$, then
$$
 |V|\leq \sum_{g\in X}|g^{G}||C_V(g)|,
$$
where  $X$ is a set of representatives for the conjugacy classes of non-central elements of prime order in $G$. 
\end{lemma}

\begin{proof}
Let $0\neq v\in V$. Now $v\in C_V(g_0)$ for some $g_0\in G\setminus Z(G)$ by Lemmas \ref{dag} and \ref{center dag}. This implies that $C_G(v)$ is a non-trivial group, so there exists  $h_0\in C_G(v)$ of prime order and $v\in C_V(h_0)$. In particular, $h_0\notin Z(G)$ by Lemma \ref{center dag}. Since  $|C_V(g)|=|C_V(h^{-1}gh)|$  for all $g,h\in G$,  the result follows.
\end{proof}

 Let $G$ be an  almost quasisimple group  where $G/Z(G)$ has socle $T$, and let  $g\in G\setminus Z(G)$. Now $\langle T,Z(G)g\rangle$ is generated by  the $T$-conjugates of $Z(G)g$, so we may define $r(g)$ to be  the minimal number of $T$-conjugates of $Z(G)g$   generating  $\langle T,Z(G)g\rangle$.

The following result appears in various incarnations in the literature; the version given here, which is essentially  \cite[Lemma 3.2]{GurTie2005}, is the one most suited to our purposes;  see also \cite[Lemma 2]{Lie1996} and the proof of \cite[Theorem 6]{HalLieSei1992}.

\begin{lemma}
\label{r(g)}
Let $G$ be an almost quasisimple group and $F$ a field. Let $V$ be a faithful  irreducible $FG$-module. Then
$$\dim_{F}( C_V(g) )\leq \dim_{F}(V)\left(1-\frac{1}{r(g)}\right )$$
for all $g\in G\setminus Z(G)$.
\end{lemma}

\begin{proof}
Let $\overline{g}$  denote the coset $Z(G)g$ for $g\in G$, and let $\overline{N}:=N/Z(G)$ where $N/Z(G)$ is the socle of $G/Z(G)$.  Fix $g\in G\setminus Z(G)$. Let $\overline{g}_1,\overline{g}_2,\ldots,\overline{g}_r$ be conjugates of $\overline{g}=\overline{g}_1$ that generate  $\langle \overline{N},\overline{g}\rangle$ where $r:=r(g)$ and the representatives $g_2,\ldots,g_r$ are chosen to be conjugates of $g$ in $G$. By this choice,  $|C_V(g)|=|C_V(g_i)|$ for $1\leq i\leq r$.  Let $W:=[V,\langle g_1,\ldots,g_r\rangle]=\Span\{[V,g_i]:1\leq i\leq r\}$. Now $W$ is spanned by $r(g)\dim_F([V,g])$ elements. Observe that $[V,N']$ is a subspace of $[V,\langle g_1,\ldots,g_r\rangle]$, for $N\leq \langle g_1,\ldots,g_r\rangle Z(G)$, and so $N'\leq \langle g_1,\ldots,g_r\rangle$.  But $1\neq N'\unlhd G$ and $V$ is faithful, so 
$[V,N']$ is a non-zero $FG$-submodule of $V$. Thus $[V,N']=[V,\langle g_1,\ldots,g_r\rangle]=V$, so $r(g)\dim_F([V,g])\geq \dim_F(V)$. Now $\dim_F(V)-\dim_F(C_V(g))\geq \dim_F(V)/r(g)$, and the result follows.
\end{proof}

The next result is a natural generalisation of part of the proof of \cite[Theorem 6]{HalLieSei1992}.

\begin{lemma}
\label{general bound}
 Let $G$ be an almost quasisimple group and $V$ a faithful irreducible $\mathbb{F}_qG$-module where $q$ is a power of a prime. If $G$ has no regular orbits on  $V$, then
$$\dim_{\mathbb{F}_q}(V)\leq r(G)\log_q{|G|},$$
where $r(G):=\max{\{r(g):g\in G\setminus Z(G)\}}$.
\end{lemma}

\begin{proof}
 By Lemmas \ref{dag}, \ref{center dag} and  \ref{r(g)},  $$q^{\dim_{\mathbb{F}_q}(V)}=|V|\leq \sum_{g\in G\setminus Z(G)}q^{\dim_{\mathbb{F}_q}(C_V(g))}\leq |G| q^{\dim_{\mathbb{F}_q}(V)\left(1-\frac{1}{r(G)}\right)}.$$ Now $q^{\dim_{\mathbb{F}_q}(V)/r(G)}\leq |G|$, and so $\dim_{\mathbb{F}_q}(V)\leq r(G)\log_q{|G|}$.
\end{proof}

Now we give some more specific bounds for the case where the socle of $G/Z(G)$ is $A_n$.

\begin{lemma}
\label{bounds}
Let $G$ be an almost quasisimple group, and suppose that the socle of $G/Z(G)$ is $A_n$ where $n\geq 5$. Let $V$ be a  faithful irreducible $\mathbb{F}_qG$-module where $q$ is a power of a prime.  If $G$ has no regular orbits on $V$, then 
\begin{equation}
\tag{1}
\label{rough bound}
\dim_{\mathbb{F}_q}(V)\leq (n-1) \log_q{|G|},
\end{equation}
and if $n\geq 7$, then
\begin{equation}
\tag{2}
\label{better bound}
\dim_{\mathbb{F}_q}(V)  \leq \max{\{(n-1) \log_q{(n(n-1)|Z(G)|)}, \tfrac{n}{2}\log_q{(2n!|Z(G)|)}\}}.\\
\end{equation}
If  $n\geq 7$ and $|Z(G)|\leq n$, then
\begin{equation}
\tag{3}
\label{even better bound}
\dim_{\mathbb{F}_q}(V) \leq \tfrac{n}{2}\log_q{(2n!|Z(G)|)}.
\end{equation}
\end{lemma}

\begin{proof}
If $g\in G\setminus Z(G)$, then  $r(g)\leq n-1$ for $n\geq 5$ by  \cite[Lemma 6.1]{GurSax2003}, so equation (\ref{rough bound}) follows from Lemma \ref{general bound}. 

Suppose that $n\geq 7$. Then $G/Z(G)=S_n$ or $A_n$.  Let $g\in G\setminus Z(G)$, and write $\overline{g}$  for the coset $Z(G)g$.  If $\overline{g}$ is not a transposition, then  $r(g)\leq n/2$ by  \cite[Lemma 6.1]{GurSax2003}. Let $S_1$ be the set of $g\in G$ for which $\overline{g}$ is a transposition, and let $S_2$ be the set of $g\in G\setminus Z(G)$ for which $\overline{g}$ is not a transposition. It follows from Lemmas \ref{dag}, \ref{center dag} and \ref{r(g)}  that
$$|V|\leq  \sum_{g\in S_1} |C_V(g)|+\sum_{g \in S_2} |C_V(g)| \leq |S_1| q^{\dim_{\mathbb{F}_q}(V)\left(1-\frac{1}{n-1}\right)}+|S_2|q^{\dim_{\mathbb{F}_q}(V)\left(1-\frac{2}{n}\right)},$$
and since $q^{\dim_{\mathbb{F}_q}(V)}=|V|$, we obtain that
$$1\leq  2\max\{ |S_1| q^{-\frac{1}{n-1}\dim_{\mathbb{F}_q}(V)},|S_2|q^{-\frac{2}{n}\dim_{\mathbb{F}_q}(V)}\}.$$
If $1\leq  2|S_1| q^{-\dim_{\mathbb{F}_q}(V)/(n-1)}$, then $\dim_{\mathbb{F}_q}(V)\leq (n-1)\log_q(2|S_1|)$. Similarly, if $1\leq 2|S_2|q^{-2\dim_{\mathbb{F}_q}(V)/n}$, then $\dim_{\mathbb{F}_q}(V)\leq (n/2)\log_q(2|S_2|)$. Thus
$$\dim_{\mathbb{F}_q}(V)\leq  \max{\{ (n-1)\log_q(2|S_1|),\tfrac{n}{2}\log_q(2|S_2|)}\}.$$
Since $2|S_1|=n(n-1)|Z(G)|$ and $|S_2|\leq |G|\leq n!|Z(G)|$, we have proved equation (\ref{better bound}).

Suppose in addition that $|Z(G)|\leq n$. First we claim that $n^5\leq 2n!$ for $n\geq 8$. Note that $(n+1)^4\leq 5n^4\leq n^5$, so if $n^5\leq 2n!$, then $(n+1)^5\leq n^5(n+1)\leq 2(n+1)!$. Thus the claim holds by induction, and so $(n(n-1)|Z(G)|)^2\leq 2n!|Z(G)|$ for $n\geq 8$. 
Now
$$ (n-1) \log_q{(n(n-1)|Z(G)|)}\leq \tfrac{n}{2}\log_q{(2n!|Z(G)|)},$$
and so $\dim_{\mathbb{F}_q}(V) \leq (n/2)\log_q{(2n!|Z(G)|)}$  when  $n\geq 8$ by equation (\ref{better bound}). Now suppose that $n=7$. It suffices to show that $(42|Z(G)|)^{12/7}\leq 2\cdot 7!|Z(G)|$ when $|Z(G)|\leq 7$, and this is true since $42^{12/7}|Z(G)|^{12/7-1}\leq 42^{12/7}7^{12/7-1}\leq 2\cdot 7!$.
\end{proof}

Motivated by equations (\ref{better bound}) and (\ref{even better bound}) of Lemma \ref{bounds}, we finish this section with a technical observation.

\begin{lemma}
 \label{decreasing}
If $C$ is an absolute constant where $C\geq 5$, then $\log_q{(C(q-1))}$ is a decreasing function in $q$ for $q\in \mathbb{R}$ and $q\geq 2$.
\end{lemma}

\begin{proof}
Let $f(q):= \log_q{(C(q-1))}$. Now $f'(q)<0$ precisely when $q\log q < (q-1)\log (C(q-1))$. Subtracting $(q-1)\log q$ from both sides, we obtain $\log q<(q-1)\log(C(q-1)/q)$, so it suffices to prove that $q<C^{q-1}(1-1/q)^{q-1}$. But $C\geq 5$ and $q\geq 2$, so $q<(C/2)^{q-1}\leq C^{q-1}(1-1/q)^{q-1}$,   as desired.
\end{proof}

\section{Symmetric and alternating groups}
\label{s: irred Sn}

Our notation for this section follows that of James \cite{JamG1978}. 
For a partition $\mu$ of $n$, let $M_F^\mu$ denote the permutation module of  $S_n$  on a Young subgroup for $\mu$ over a field $F$, and let $S_F^\mu$ denote the \textit{Specht module} for $\mu$ over $F$, which is the submodule of $M_F^\mu$ spanned by the polytabloids. Let $<,>$  denote the unique  $S_n$-invariant symmetric non-degenerate bilinear form on $M_F^\mu$ for which the natural basis of $M_F^\mu$ is orthonormal, and 
write $S^{\mu\perp}_F$ for the orthogonal complement  of $S^\mu_F$ with respect to this form. Define
$$ D^\mu_F:=S_F^\mu/(S_F^\mu\cap S_F^{\mu\perp}).$$
When context permits, we omit the subscript $F$ and write $M^\mu$, $S^\mu$, or $D^\mu$. 

 It is well known that for a field $F$ of characteristic $p$, the  $D^\mu_F$  afford a complete list of non-isomorphic irreducible $FS_n$-modules as $\mu$ ranges over the $p$-regular partitions of $n$   \cite[Theorem 11.5]{JamG1978}; recall that  a partition $\mu$ is \textit{$p$-regular} for $p$ prime if no part of $\mu$ is repeated $p$ times, and  always $0$-regular for convenience. In particular, when $p>n$ (or when $p=0$), the   $S^\mu_F$  afford a complete list of non-isomorphic irreducible $FS_n$-modules as $\mu$ ranges over the partitions of $n$. 
Every field is a splitting field for $S_n$ \cite[Theorem 11.5]{JamG1978}, and every field  containing $\mathbb{F}_{p^2}$ for $p$ prime is  a splitting field for $A_n$ 
(cf.  \cite[Corollary 5.1.5]{Maa2011} or \cite{Mey2004}).

For each $p$-regular partition $\mu$ of $n$, there exists a unique $p$-regular partition $\lambda$ for which $D^\lambda\simeq D^\mu\otimes_{F}\sgn$, and we denote this partition by $m(\mu)$. 
Note that $m(\mu)=\mu$ when $p=2$, so we omit $m(\mu)$ from Table \ref{tab: Sn ex} below. 
Moreover, given  an irreducible $FA_n$-module $V$,   there exists a $p$-regular partition $\mu$ for which $V\leq D^\mu\downarrow A_n$ by Lemma \ref{H to G}. 

 In this section, we prove the following theorem.
  
  \begin{thm}
\label{regular}
Let $H$ be  $S_n$ or $A_n$ where $n\geq 5$. Let $G$ be a group for which  $H\leq G\leq H\times \mathbb{F}_p^*$ where $p$ is a prime and $p\leq n$. Let $V$ be a faithful irreducible $\mathbb{F}_pH$-module and  $\mu$ a $p$-regular partition  of $n$ for which  $V\leq D^\mu\downarrow H$. Let $d:=\dim_{\mathbb{F}_p}(V)$.
 \begin{itemize}
 \item[(i)] If either $\mu$ or $m(\mu)$ is $(n-1,1)$, then  $G$ has a regular orbit on $V$ if and only if $G=A_n$ and $p=n-1$.
\item[(ii)] If neither $\mu$ nor $m(\mu)$ is $(n-1,1)$, then  $G$ has a regular orbit on $V$ if and only if  $(n,p,\mu,G,d)$  is not listed in Table $\ref{tab: Sn ex}$.
\end{itemize}
\end{thm}

  \begin{table}[!h]
\renewcommand{\baselinestretch}{1.1}\selectfont
\centering
\begin{tabular}{ l l l l l l }
\hline
$n$ & $p$ & $\mu$ & $G$ &   $d$ & $m(\mu)$ \\
\hline 
  5 & 2  & (3,2) &   $A_5$, $S_5$ & 4  & - \\
   & 3 & (3,1,1) & $A_5\times \mathbb{F}_3^*$, $S_5\times \mathbb{F}_3^*$ & 6 & (3,1,1) \\

  6 & 2 & (4,2) &  $A_6$, $S_6$  & 4  &  - \\

  & 3 & (4,1,1) &  $H\in\{A_6,S_6\}$  & 6 & (4,1,1) \\

  & 5& (3,3) & $H\in\{A_6,S_6\}$, $G\neq A_6$  & 5  & (2,2,2)  \\

 &  & (2,2,2) & $H\in\{A_6,S_6\}$, $G\neq A_6$ & 5   & (3,3)  \\

 7 & 2     &(4,3) & $A_7$ & 4   & - \\
  & & &   $S_7$  & 8  & - \\
& & (5,2) & $S_7$  & 14 & - \\

8 & 2   &(5,3) & $A_8$ & 4  &- \\ 
& & & $S_8$ & 8   & - \\
& &(6,2) &  $A_8$, $S_8$ & 14  &-  \\

 9 & 2  &(5,4) & $A_9$ & 8  & - \\
& & &  $S_9$ & 16  & - \\

&  & (5,3,1) & $A_9$ & 20 & - \\

10 & 2  &(6,4) &  $A_{10}$, $S_{10}$ & 16  &  -\\

12 & 2  &(7,5) & $S_{12}$ & 32  &-   \\
\hline
\end{tabular}
\caption{$\mathbb{F}_pG$-modules $V$ on which $G$ has no regular orbits}
\label{tab: Sn ex}
\renewcommand{\baselinestretch}{1.3}\selectfont
\end{table}

For a partition $\mu$ of $n$, the dimension of $D^\mu$ is the rank of the Gram matrix with respect to a basis of $S^\mu$.
 However, there is no formula that computes this rank in general, in contrast to  the Specht module $S^\mu$, whose dimension is  given by the characteristic-independent hook formula \cite[Theorem 20.1]{JamG1978}. Thus we require  lower bounds for the dimension of $D^\mu$. These we obtain using  a method of James \cite{JamG1983}, which requires the following notation. 

Let $F$ be a field of characteristic $p$. 
For each non-negative integer $m$,   write $R_n(m)$ for the class of irreducible $FS_n$-modules $V$ such that for some $p$-regular partition $\mu$ of $n$,
\begin{itemize}
\item[(i)] $\mu_1\geq n-m$ where $\mu_1$ is the largest part of $\mu$, and
\item[(ii)] $V\simeq D^\mu$ or $V\simeq D^\mu\otimes_{F}\sgn$.
\end{itemize}
Now \cite[Lemma 4]{JamG1983} and \cite[Appendix Table 1]{JamG1983} enable us to construct  functions $f(n)$ with the property that for every irreducible $FS_n$-module $V$,  either $V\in R_n(2)$ or $\dim_F(V)> f(n)$ (cf.\ Lemma \ref{key}). 

Thus the proof of Theorem \ref{regular} divides into two cases. Suppose we are given a faithful irreducible $\mathbb{F}_pS_n$-module $V$ on which $G$ has no regular orbits. If $V\notin R_n(2)$, then $\dim_{\mathbb{F}_p}(V)$ is bounded below by  $f(n)$ and above by functions of \S \ref{s: bounds}, and this is usually a contradiction. Otherwise  $V\in R_n(2)$, in which case the functions of \S \ref{s: bounds} are useless since $\dim_{\mathbb{F}_p}(V)\leq n^2$, so we use constructive methods instead. Note that for a field $F$, the only non-faithful irreducible $FS_n$-modules are the trivial module $D^{(n)}$ and the sign module $D^{(n)}\otimes_F\sgn$, 
 and so an irreducible $FS_n$-module $V$ is faithful if and only if $V\notin R_n(0)$.

We will often make use of the known Brauer character tables of the symmetric and alternating groups. The Brauer Atlas \cite{BAtlas} contains the Brauer character tables of $S_n$ and $A_n$ for $n\leq 12$ and $p\leq n$, while {\sf GAP} \cite{GAP4} in conjunction with the {\sc SpinSym} package \cite{SpinSym} contains the Brauer character tables of $S_n$ and $A_n$ for $n\leq 17$ and $p\leq n$, as well as $n=18$ when $p=2,3,5$ or $7$, and $n=19$ when $p=2$. 
Moreover, for those character tables in \cite{GAP4},  {\sc  SpinSym} provides a function to determine the corresponding partitions. 

\begin{remark}
\label{rem: dim}
If $H=S_n$ or $A_n$  and $(n,p,G,d)$ is listed in Table $\ref{tab: total ex}$, then $(n,p,\mu,G,d)$ is listed in Table $\ref{tab: Sn ex}$  by \cite{BAtlas}. Hence Theorem \ref{regular total} follows from Theorem \ref{regular} for such $H$.
\end{remark}

This section is organised as follows. In \S \ref{s: not Rn(2)} we consider modules that are not in $R_n(2)$,  and in \S \ref{s: Rn(2)} and \S \ref{s: Rn(1)} we consider modules that are in $R_n(2)\setminus R_n(1)$ and $R_n(1)$ respectively. Lastly, in \S \ref{s: sym proof} we prove Theorem \ref{regular}.

\subsection{Modules not in $R_n(2)$}
\label{s: not Rn(2)}

 The following lemma is the key tool for  this case. It relies significantly on  \cite{JamG1983}. We include the case $p>n$ for completeness.

\begin{lemma}
\label{key}
Let $F$ be a field of positive characteristic $p$. Let $V$ be an irreducible $FS_n$-module where $n\geq 15$ when $p=2$ and $n\geq 11$ when $p$ is odd. 
 Let $$f(n):=\tfrac{1}{6}(n^3-9n^2+14n-6).$$ For $p=2$, let $f_p(n)$ be defined by $f_p(n)=f(n)$ for $n\geq 23$ and 
\begin{align*}
 f_p(15)&=f_p(16)=127,\\
f_p(17)&=f_p(18)=253,\\
f_p(19)&=f_p(20)=505,\\
f_p(21)&=f_p(22)=930.
\end{align*}
For odd $p$, let $f_p(n)$ be defined by $f_p(n)=f(n)$ for $n\geq 16$ and 
\begin{align*}
 f_p(11)&=54, \\
f_p(12)&=88, \\
f_p(13)&=107,\\
f_p(14)&=175,\\
f_p(15)&=213.
\end{align*}
Then $V\in R_n(2)$ or $\dim_F(V)>f_p(n)$. 
\end{lemma}

\begin{proof} 
 Suppose that there is a function $g:\mathbb{N}\to\mathbb{R}$  and a positive integer $N$ for which: 

(i) $2g(n)>g(n+2)$ for all $n\geq N$.

(ii) For $n=N$ or $N+1$, if $U$ is an irreducible $FS_n$-module, then $U\in R_n(2)$ or  $\dim_{F}(U)>g(n)$.

(iii) For all $n\geq N$, if $U\in R_n(4)\setminus R_n(2)$, then $\dim_{F}(U)>g(n)$.

\noindent Then  \cite[Lemma 4]{JamG1983} implies that for all $n\geq N$, either $V\in R_n(2)$ or $\dim_{F}(V)>g(n)$. Thus it suffices to show that $f_p(n)$ satisfies  conditions (i)-(iii) with $N=15$ when $p=2$ and $N=11$ otherwise. Note that $2f_2(n)>f_2(n+2)$ for all $n\geq 15$, and if $p$ is odd, then $2f_p(n)>f_p(n+2)$ for all $n\geq 11$.  Moreover, using the lower bounds of  \cite[Appendix Table 1]{JamG1983}, it is routine to verify that if $U\in R_n(4)\setminus R_n(2)$ and  $n\geq 11$, then $\dim_{F}(U)>f(n)$ unless $U$ is  $D^{(7,4)}$ or its associate, in which case $\dim_{F}(U)\geq 55>f_p(11)$ for all odd $p$. Since $f(n)\geq f_p(n)$ for all $p$ and $n\geq 11$, it remains to check condition (ii).

Let $U$ be an irreducible $FS_n$-module, and suppose that $U$ is not in $R_n(2)$. To begin, suppose that $p=2$. If $n=15$ or $16$, then  $\dim_{F}(U)>(n-1)(n-2)/2$  by \cite[Theorem 7]{JamG1983} since $U\not\in R_n(2)$. Using  the Brauer character table of $S_n$ \cite{GAP4}, we  check that  $\dim_{F}(U)\geq 128>f_2(n)$. Thus condition (ii) holds with  $N=15$. 

Now suppose that $p$ is an odd prime and  $n=11$ or $12$.  First assume  that $p\leq n$. Since   $\dim_{F}(U)>(n-1)(n-2)/2$ by \cite[Theorem 7]{JamG1983},  $\dim_{F}(U)\geq 55$ when $n=11$ and $\dim_{F}(U)\geq 89$ when $n=12$ by  \cite{BAtlas}. Thus $\dim_{F}(U)>f_p(n)$, as desired. Assume instead that $p>n$. Now $U\simeq S^\mu$ for some partition $\mu$ of $n$. The dimensions of the Specht modules are listed in the decomposition matrices in \cite[Appendix]{JamG1978}: $\dim_{F}(U)\geq 55$ when $n=11$ and $\dim_{F}(U)\geq 89$ when $n=12$. Thus condition (ii) holds with $N=11$.
\end{proof}

Note that  the dimension of $D_F^{(n-3,3)}$ for
 a field $F$ of positive characteristic  is precisely $f(n)+1$ for infinitely many $n$ by  \cite[Appendix Table 1]{JamG1983}, so Lemma \ref{key} provides a tight lower bound for $\dim_F(V)$ for $V\notin R_n(2)$.
 
Let $F$ be an arbitrary field.
By \cite[Theorem 5]{JamG1983}, there are only finitely many $n$ for which $D^\mu\not\in R_n(3)$ and $\dim_{F}(D^\mu)\leq n^3$. Motivated by classifying  these exceptional modules, M\"uller \cite{Mul2011} determined the dimensions of the irreducible $FS_n$-modules of dimension at most $n^3$ for  $\Char(F)\in\{2,3\}$ along with the corresponding partitions; we will use this information whenever character tables are not available.

We begin with a reduction for almost quasisimple groups $G$ with $F^*(G)'\simeq A_n$.

\begin{lemma}
\label{An reduction}
Let $G$ be an almost quasisimple group where   $N:=F^*(G)'\simeq A_n$ and $n\geq 11$. Let $F$ be a finite field. Let $V$ be a faithful irreducible $FG$-module,  $k:=\End_{FG}(V)$ and  $q:=|k|$. Let $W$ be an irreducible $kN$-submodule of $V$ and $\mu$ a $\Char(F)$-regular partition of $n$ for which $W\leq D^\mu\downarrow N$.
If $G$ has no regular orbits on $V$ and $D^\mu\notin R_n(2)$, then $n\leq 20$, and if $\Char(F)\leq n$ and $q$ is odd, then $(n,q)=(11,5)$ and  $\dim_k(W)=\dim_k(D^\mu)=55$.
Moreover, if $n\geq 15$ and $q$ is even, then $(n,\mu,q,\dim_k(W),\dim_k(D^\mu))$ is listed in Table $\ref{tab: reduction}$.
\begin{table}[!ht]
\centering
\begin{tabular}{ l l l l l}
\hline
$n$ &  $\mu$ & $q$  & $\dim_k(W)$  & $\dim_k(D^\mu)$ \\
\hline

$15$ & $(8,7)$ & $2$, $4$, $8$, $16$, $32$  &  $64$ & $128$  \\

$16$ &  $(9,7)$ & $2$, $4$, $8$, $16$, $32$, $64$   & $64$ & $128$  \\
& $(13,3)$ & $2$  & $336$ & $336$ \\

$17$ & $(9,8)$ & $2$, $4$, $8$   & $128$ & $256$ \\

  $18$ & $(10,8)$& $2$  & $256$  & $256$  \\
$19$ & $(10,9)$& $2$  & $512$ & $512$ \\
 
& & $4$  & $256$  & $512$   \\
 
 $20$ & $(11,9)$ & $2$   & $512$ & $512$ \\
 
 & & $4$   & $256$ & $512$  \\

\hline
\end{tabular}
\caption{Possible $\dim_k(W)$ and $\dim_k(D^\mu)$ when  $n\geq 15$ and $q$ is even}
\label{tab: reduction}
\end{table}
\end{lemma}

In fact, $\dim_k(V)=\dim_k(W)$ or $2\dim_k(W)$ by Lemma \ref{index 2}, for $W$ is an irreducible $kF^*(G)$-submodule of $V$ by Lemmas \ref{quasi} and \ref{scalar}, and $[G:F^*(G)]\leq 2$ by Lemma \ref{quasi}.

\begin{proof}[Proof of Lemma $\ref{An reduction}$]
Suppose that $G$ has no regular orbits on $V$ and $D^\mu\notin R_n(2)$.
Let $p:=\Char(F)$.
Since $V$ is a faithful absolutely irreducible  $kG$-module, Lemma \ref{scalar} implies that $Z(G)\leq k^*$, and so $|Z(G)|\leq q-1$. Let
$$g(q,n):=\max{\{(n-1) \log_q{(n(n-1)(q-1))}, \tfrac{n}{2}\log_q{(2n!(q-1))}\}}.$$
Now equation (\ref{better bound}) of Lemma \ref{bounds} implies that $\dim_{k}(V)\leq  \lfloor g(q,n) \rfloor$.  Since $\dim_k (D^\mu)$ is equal to $\dim_k(W)$ or $2\dim_k(W)$ by Lemma \ref{index 2}, it follows that $\dim_k(D^\mu)\leq 2\lfloor g(q,n) \rfloor$. Note that if $n$ is fixed, then $g(q,n)$ is a decreasing function in $q$ by Lemma \ref{decreasing}.  

To begin, suppose that $q$ is odd. Recall the function $f_p(n)$  defined in Lemma \ref{key}. Since $n\geq 11$ and $D^\mu\notin R_n(2)$ by assumption, it follows from this lemma that $f_p(n)<\dim_{k}(D^\mu)$. Thus $f_p(n)<2\lfloor g(q,n) \rfloor$. However, if $n\geq 21$, then $2\lfloor g(q,n)\rfloor \leq 2\lfloor g(3,n)\rfloor \leq f_p(n)$, a contradiction. Thus  $n\leq 20$ for odd $q$, as claimed. Similarly, if $q\geq 121$, then we obtain a contradiction for all $n\geq 11$, so $q<121$. Moreover, if $q\geq 5$, then $n\leq 15$; if $q\geq 9$, then $n\leq 14$; if $q\geq 11$, then $n\leq 13$; if $q\geq 25$, then $n\leq 12$; and if $q\geq 27$, then $n\leq 11$. 

Hence if we assume that $p\leq n$, then $(n,q)$ is listed in Table \ref{tab: q odd}.
Suppose that $q=3$ and $n=19$ or 20. Since $\dim_k(D^\mu)\leq 2g(3,n)\leq n^3$, we apply   \cite{Mul2011} to determine the dimensions of those $D^\mu$ for which $f_p(n)<\dim_k(D^\mu)\leq 2\lfloor g(3,n)\rfloor $. However, there are no such $D^\mu$ when $n=20$, and when $n=19$, the only possible dimension is $647$, in which case 
$W=D^\mu\downarrow N$ since $647$ is odd,  so $\dim_k(D^\mu)\leq \dim_k(V)\leq \lfloor g(3,19)\rfloor=352$, a contradiction. Similarly, for each remaining $(n,q)$ besides $(11,5)$, we use the Brauer character tables in \cite{SpinSym, GAP4, BAtlas} to determine that if there exists a $p$-regular partition $\mu$ for which $f_p(n)<\dim_k(D^\mu)\leq 2\lfloor g(3,n)\rfloor $, then $W=D^\mu\downarrow N$ and $\lfloor g(q,n)\rfloor<\dim_k(D^\mu)$, a contradiction. Thus $(n,q)=(11,5)$, and by a similar argument,  $\dim_k(W)=\dim_k(D^\mu)=55$. 

\begin{table}[!ht]
\centering
\begin{tabular}{ l l }
\hline
$n$ & $q$   \\
\hline 
 
11 & 3, 5, 7, 9, 11, 25, 27, 49, 81 \\

12 & 3, 5, 7, 9, 11, 25 \\

13 & 3, 5, 7, 9, 11, 13 \\

14 & 3,  5, 7, 9 \\

15 & 3, 5, 7 \\

$16\leq n\leq 20$ & 3 \\

\hline
\end{tabular}
\caption{Possible odd $q$ when $n\geq 11$}
\label{tab: q odd}
\end{table}

We may assume for the remainder of the proof that $q$ is even and $n\geq 15$. Recall the function $f_2(n)$  defined in Lemma \ref{key}. As  for odd $q$, it follows that $f_2(n)<\dim_{k}(D^\mu)$, and so $f_2(n)<2\lfloor g(q,n)\rfloor$. However, if $n\geq 31$, then $2\lfloor g(q,n)\rfloor\leq 2\lfloor g(2,n)\rfloor\leq f_2(n)$, and if $n\geq 21$ and $q\geq 4$, then $2\lfloor g(q,n)\rfloor \leq 2\lfloor g(4,n)\rfloor\leq f_2(n)$, both contradictions. Thus either $n\leq 20$, or $q=2$ and $21\leq n\leq 30$.

Suppose for a contradiction that $q=2$ and $21\leq n\leq 30$. Since $\dim_k(D^\mu)\leq 2g(2,n)\leq n^3$, we apply  \cite{Mul2011} to determine the dimensions of those $D^\mu$ for which $f_2(n)<\dim_k(D^\mu)\leq 2\lfloor g(2,n) \rfloor$; these are listed in Table \ref{tab: big n}. Moreover, if $\dim_k(D^\mu)\neq 1024$, then $\mu=(n-3,3)$, and if $\dim_k(D^\mu)=1024$, then $\mu=(11,10)$ or $(12,10)$. If $(n,\dim_k(D^\mu))\neq (21,1024)$, then $W=D^\mu\downarrow N$ by \cite[Theorem 1.1]{Ben1988}, so $\dim_k(D^\mu)\leq \dim_k(V)\leq \lfloor g(2,n)\rfloor$. However, it can be verified that $\lfloor g(2,n)\rfloor<\dim_k(D^\mu)$ in each case, a contradiction. Similarly, if  $n=21$ and $\dim_k(D^\mu)=1024$, then $D^\mu\downarrow A_n$  is irreducible over $k=\mathbb{F}_2$ by \cite[Theorems 5.1 and 6.1]{Ben1988}, in which case $\dim_k(D^\mu)\leq \dim_k(V)\leq \lfloor g(2,21)\rfloor=697$, a contradiction.

\begin{table}[!ht]
\centering
\begin{tabular}{ l | l l l l l}
\hline
$n$ & $21$, $22$ & $23$, $24$ & $25$, $26$ & $28$ & $30$ \\
$\dim_k(D^\mu)$ & 1024, 1120 & 1496 & 2000 & 2548 & 3248 \\
\hline 
\end{tabular}
\caption{Possible $\dim_k(D^\mu)$ when $q=2$ and $n\geq 21$}
\label{tab: big n}
\end{table} 

Thus $q$ is even and $15 \leq n\leq 20$. Note that if $q\geq 128$, then  $2\lfloor g(q,n)\rfloor \leq 2\lfloor g(128,n)\rfloor \leq f_2(n)$, a contradiction.  Moreover,   if $q\geq 8$, then $n\leq 18$; if $q\geq 16$, then $n=18$ or $n\leq 16$; if $q\geq 32$, then $n\leq 16$; and if $q\geq 64$, then $n=16$. Hence $(n,q)$ is listed in  Table \ref{tab: q even}.
\begin{table}[!ht]
\centering
\begin{tabular}{ l l }
\hline
$n$ & $q$   \\
\hline 
15 & 2, 4, 8, 16, 32 \\
16 & 2, 4, 8, 16, 32, 64 \\
17 & 2, 4, 8\\
18 & 2, 4, 8, 16 \\
  19, 20 & 2, 4 \\
\hline
\end{tabular}
\caption{Possible even $q$ when $15\leq n\leq 20$}
\label{tab: q even}
\end{table}
    
First suppose that $n=20$. Since $\dim_k(D^\mu)\leq 2g(2,n)\leq n^3$, we apply \cite{Mul2011} to determine that the only $D^\mu$ for which $f_2(n)<\dim_k(D^\mu)\leq 2\lfloor g(q,n)\rfloor$ are those with dimension $512$ or $780$ when $q=2$ and dimension $512$ when $q=4$. Moreover, if $\dim_k(D^\mu)=512$, then $\mu=(11,9)$, and if $\dim_k(D^\mu)=780$, then $\mu=(17,3)$.  
If $q=2$ and $\dim_k(D^\mu)=780$, then $W=D^\mu\downarrow N$ by \cite[Theorem 1.1]{Ben1988}, and so $\dim_k(D^\mu)\leq \dim_k(V)\leq \lfloor g(2,20)\rfloor=620$, a contradiction. Thus $\dim_k(D^\mu)=512$ and $q=2$ or $4$, in which case $D^\mu\downarrow A_n$  is irreducible if and only if $q=2$ by \cite[Theorems 5.1 and 6.1]{Ben1988}. Thus either $q=2$ and $\dim_k(W)=512$, or $q=4$ and $\dim_k(W)=256$.

Similarly,  using the Brauer character tables in \cite{SpinSym, GAP4, BAtlas}, we determine for each remaining $(n,q)$ in Table \ref{tab: q even} that if there exists a $2$-regular partition $\mu$ for which $f_2(n)<\dim_k(D^\mu)\leq 2\lfloor g(q,n)\rfloor$ and either $D^\mu\downarrow N$ splits, or $W=D^\mu\downarrow N$ and $\lfloor g(q,n)\rfloor\geq \dim_k(D^\mu)$, then $(n,\mu,q,\dim_k(W),\dim_k(D^\mu))$ is listed in Table \ref{tab: reduction}. 
\end{proof}

Now we are in a position to determine the regular orbits of $S_n\times\mathbb{F}_p^*$ on $\mathbb{F}_pS_n$-modules not in $R_n(2)$. We also prove some results for $\mathbb{F}_pS_n$-modules  in $R_n(2)\setminus R_n(1)$ when $n$ is small, as the inclusion of these cases   simplifies the proof.   

\begin{prop}
\label{regular Sn}
Let $G$ be a group for which $S_n\leq G\leq S_n\times \mathbb{F}_p^*$  where  $n\geq 7$ and $p$ is a prime such that $p\leq n$. Let $\mu$ be a $p$-regular partition of $n$ and $V$ the $\mathbb{F}_pS_n$-module $D^\mu$.
\begin{itemize}
\item[(i)] If $D^\mu\notin R_n(2)$, then $G$ has no regular orbits on $V$ if and only if $p=2$ and $\mu=(\lfloor n/2\rfloor+1,\lfloor (n-1)/2\rfloor )$ for $7\leq n\leq 10$ or $n=12$.
\item[(ii)] If $D^\mu\in R_n(2)\setminus R_n(1)$ where either $n\leq 11$, or $12\leq n\leq 14$ and $p=2$, then $G$ has no regular orbits on $V$ if and only if $p=2$ and $\mu=(n-2,2)$ for $7\leq n\leq 8$. 
\end{itemize}
\end{prop}

\begin{proof}
We will prove (i) and (ii) simultaneously. Therefore,  we will  assume throughout this proof that either  $D^\mu\notin R_n(2)$, or $D^\mu\in R_n(2)\setminus R_n(1)$ and either $n\leq 11$, or $12\leq n\leq 14$ and $p=2$. In particular,  $D^\mu$ is faithful. Note that  $\End_{\mathbb{F}_pG}(V)=\mathbb{F}_p$.

  Suppose that $G$ does not have a regular orbit  on $V$. First consider the case where $p=2$ and $n\geq 15$.  Lemma \ref{An reduction} implies that $\mu$ and $\dim_{\mathbb{F}_2}(D^\mu)$ are listed in Table \ref{tab: reduction}. If $\mu$ is   $(9,7)$, $(13,3)$,  $(10,8)$ or $(11,9)$,
  then using {\sc Magma} (cf. \S \ref{s: comp} for further details), we determine that $S_n\times \mathbb{F}_p^*$ has a regular orbit on $V$, a contradiction. Otherwise,  $\mu$ is $(8,7)$, $(9,8)$ or $(10,9)$, in which case $D^\mu=D^\lambda\downarrow S_n$ where $\lambda$ is $(9,7)$,  $(10,8)$ or $(11,9)$ respectively by \cite[Theorem 9.3]{JamG1978}, so $G$ has a regular orbit on $V$, a contradiction.

  Next suppose that either $p=2$ and $n\leq 14$, or $p$ is odd.
  We claim that $(n,p,\dim_{\mathbb{F}_p}(D^\mu))$ is listed in Table \ref{tab: all dim}. 
   Note that if $p$ is odd, then $n\leq 11$ by Lemma \ref{An reduction}.
 Let
$$g(p,n):= \tfrac{n}{2}\log_p{(2n!(p-1))}.$$
Since $Z(G)\leq \mathbb{F}_p^*$ by Lemma \ref{scalar} and $p\leq n$,  equation (\ref{even better bound}) of Lemma \ref{bounds} implies that $\dim_{\mathbb{F}_p}(V)\leq \lfloor g(p,n)\rfloor $.  Note that if $U$ is an irreducible $\mathbb{F}_pS_n$-module such that $U\in R_n(1)$ but $U$ does not have dimension 1, then  $\dim_{\mathbb{F}_p}(U)$ is either  $n-2$ when $p\mid n$, or $n-1$ when $p\nmid n$ (cf.\ \S \ref{s: Rn(1)}).
Hence by \cite{BAtlas,GAP4}, the dimensions of those $D^\mu$ for which $D^\mu\notin R_n(1)$ and $\dim_{\mathbb{F}_p}(V)\leq \lfloor g(p,n)\rfloor$ are precisely those listed in Table \ref{tab: all dim}, proving the claim.

\begin{table}[!h]
\centering
\begin{tabular}{ l l l  }
\hline
$n$ &  $p$  & $\dim_{\mathbb{F}_p}(D^\mu)$   \\
\hline
7 & 2 &  $8^\times$, $14^\times$, 20 \\
& 3 &  $13$,  $15$, $20$\\
 & 5 &  $8$,  $13$, $15$,  $20$ \\
  & 7 &  $10$,  $14$\\

8  & 2 &  $8^\times$, $14^\times$, 40, 64 \\
 & 3 &  $13$, $21$, $28$,  $35$ \\
& 5 &   $13$,  $20$,  $21$\\
 & 7 &   $14$,  $19$,  $21$ \\

9 & 2 &  $16^\times$, 26, 40, 48, 78 \\
 & 3 &   $21$,  $27$,  $35$,  $41$ \\
& 5 &   $21$, $27$, $28$,  $34$\\
 & 7 &  $19$,  $28$ \\

 \hline
\end{tabular}
\quad
\begin{tabular}{ l l l  }
\hline
$n$ & $p$  & $\dim_{\mathbb{F}_p}(D^\mu)$     \\
\hline 
10 & 2 &  $16^\times$, 26, 48  \\
 & 3 &  $34$,  $36$,  $41$  \\
 & 5 &  $28$, $34$, $35$ \\

  & 7 &  $35$,  $36$,  $42$ \\

11 & 2 &  32, 44,100,144 \\
& 3 &  $34$,  $45$ \\
 & 5 &  $43$, $45$,  $55$ \\
 & 7 &   $44$,  $45$   \\
 & 11 &  $36$,  $44$ \\
 
 12 & 2 & $32^\times$, 44, 100, 164\\
13 & 2 &  $64$,  208\\ 
14 & 2 &  $64$,  208  \\
\hline
\end{tabular}
\caption{Possible $\dim_{\mathbb{F}_p}(D^\mu)$}
\label{tab: all dim}
\renewcommand{\baselinestretch}{1.3}\selectfont
\end{table}

Suppose that  $\dim_{\mathbb{F}_p}(D^\mu)$ is listed in Table \ref{tab: all dim} with no adjacent $^\times $. 
Using {\sc Magma}, we determine that $S_n\times \mathbb{F}_p^*$ has a regular orbit on $V$,  a contradiction. All of these computations are routine except for  $n=12$ and $\dim_{\mathbb{F}_2}(V)=44$; in this  case we use  {\sc Orb} \cite{ORB}. Also, we do not require {\sc Magma}  when $p=2$ and $(n,\dim_{\mathbb{F}_p}(D^\mu))$ is one of $(11,44)$, $(11,100)$, or $(13,208)$, for in these cases  $D^\mu=D^\lambda\downarrow S_{n}$ where $\lambda$ is a $p$-regular partition  of $n+1$  such that $\dim_{\mathbb{F}_p}(D^\lambda)$ is  listed in Table \ref{tab: all dim} by \cite[Theorem 9.3]{JamG1978}.

Thus $\dim_{\mathbb{F}_p}(D^\mu)$ is listed in Table \ref{tab: all dim} with an adjacent $^\times $.  From the decomposition matrices in \cite{JamG1978},    $\mu=(\lfloor n/2\rfloor+1,\lfloor (n-1)/2\rfloor )$ when $\dim_{\mathbb{F}_2}(D^\mu)=2^{\lfloor (n-1)/2\rfloor}$ for $7\leq n\leq 10$ or $n=12$, and  $\mu=(n-2,2)$ when $\dim_{\mathbb{F}_2}(D^\mu)=14$ for $7\leq n\leq 8$, as desired. 

Conversely, suppose that $p=2$ and $\mu$ is either  $(n-2,2)$ for $7\leq n\leq 8$, or  $(\lfloor n/2\rfloor+1,\lfloor (n-1)/2\rfloor )$ for $7\leq n\leq 10$ or $n=12$. Note that $G=S_n$.
If $\mu=(6,2)$ or $(\lfloor n/2\rfloor+1,\lfloor (n-1)/2\rfloor )$ for $7\leq n\leq 10$, then $|V|<|G|$, so $G$ has no regular orbits on $V$.
If $\mu=(5,2)$, then no orbit is regular by {\sc Magma}, and if $\mu=(7,5)$, then no orbit is regular by {\sc Orb} \cite{ORB}.
\end{proof}

Now we consider the regular orbits of $A_n\times\mathbb{F}_p^*$.

\begin{prop}
\label{regular An}
Let $G$ be a group for which $A_n\leq G\leq A_n\times \mathbb{F}_p^*$  where  $n\geq 7$ and $p$ is a prime such that $p\leq n$. Let $V$ be a faithful irreducible $\mathbb{F}_pA_n$-module, and let $\mu$ be a $p$-regular partition of $n$ for which  $V\leq D^\mu\downarrow A_n$. Suppose that $D^\mu \notin R_n(2)$.
\begin{itemize}
 \item[(i)] If $V\neq D^\mu\downarrow A_n$, then $G$ has no regular orbits on $V$ if and only if $p=2$ and $\mu $ is $(5,3,1)$ or $(\lfloor n/2\rfloor+1,\lfloor (n-1)/2\rfloor )$ for $7\leq n\leq 9$.  
 \item[(ii)] If    $V= D^\mu\downarrow A_n$, then $G$ has no regular orbits on $V$ if and only if $p=2$ and $\mu=(6,4)$. 
 \end{itemize}
 \end{prop}

\begin{proof}
(i) By Lemma \ref{index 2},  $D^\mu\downarrow A_n= V\oplus Vg$ for every $g\in S_n\setminus A_n$. Now $\End_{\mathbb{F}_pG}(V)=\mathbb{F}_p$ since for every field $F$ of characteristic $p$, the irreducible $FS_n$-module $D^\mu \otimes_{\mathbb{F}_p}F$ restricted to $A_n$ is $(V\otimes_{\mathbb{F}_p}F)\oplus (V\otimes_{\mathbb{F}_p}F)g$, and so $V\otimes_{\mathbb{F}_p}F$ must be irreducible by Lemma \ref{index 2}. Note that $G$ has a regular orbit on $V$ if and only if $G$ has a regular orbit on $Vg$.

Suppose that $G$ does not have a regular orbit on $V$. We claim that $(n,p,\dim_{\mathbb{F}_p}(D^\mu))$ is listed in Table \ref{tab: An mag}. 
First suppose that $p=2$ and $n\geq 15$. 
Lemma \ref{An reduction} implies that $\dim_{\mathbb{F}_2}(D^\mu)$ is listed in Table \ref{tab: reduction}  where $q=2$, $W=V$ and $k=\mathbb{F}_2$, so the claim holds. 

\begin{table}[!h]
\centering
\begin{tabular}{ l  l  l  l l l l l l l ll }
\hline
$(n,p)$ & (7,2)  & (8,2) & (9,2) & (9,5) & (10,2) & (10,5)  & (15,2) & (16,2) & (17,2) \\
$\dim_{\mathbb{F}_p}(D^\mu)$ & 8$^\times$  & 8$^\times$, 40 & 16$^\times$, 40$^\times$ & 70 & 128 & 70 & 128 & 128 & 256 \\
\hline
\end{tabular}
\caption{Possible  $\dim_{\mathbb{F}_p}(D^\mu)$ when $V\neq  D^\mu\downarrow A_n$}
\label{tab: An mag}
\end{table}

Now suppose that either $p=2$ and $n\leq 14$, or $p$ is odd. Note that if $p$ is odd, then $n\leq 11$ by Lemma \ref{An reduction}. Let $$h(p,n):=n\log_{p}(n!(p-1)/2).$$ 
 By Lemma \ref{scalar}, $Z(G)\leq \mathbb{F}_p^*$, so Lemma \ref{general bound} implies that $\dim_{\mathbb{F}_p}(V)\leq r(G)\log_{p}(n!(p-1)/2)$. Since $r(G)=r(A_n)\leq n/2$ by  \cite[Lemma 6.1]{GurSax2003}, and since $\dim_{\mathbb{F}_p}(V)=\dim_{\mathbb{F}_p}(D^\mu)/2$, we obtain that $\dim_{\mathbb{F}_p}(D^\mu)\leq\lfloor h(p,n)\rfloor$.  By \cite{BAtlas,GAP4,SpinSym},  the dimensions of those $D^\mu$ for which $D^\mu\downarrow A_n$ splits (over $\mathbb{F}_p$) are precisely those listed in Table \ref{tab: An mag}, proving the claim.
 
Suppose that  $\dim_{\mathbb{F}_p}(D^\mu)$ is listed in Table \ref{tab: An mag} with no adjacent $^\times $. Using {\sc Magma}, we determine that $A_n\times \mathbb{F}_p^*$ has a regular orbit on $V$,  a contradiction.
Thus $\dim_{\mathbb{F}_p}(D^\mu)$ is listed in Table \ref{tab: An mag} with an adjacent $^\times $.  Using the decomposition matrices in \cite{JamG1978}, we determine that   $\mu$   is $(5,3,1)$ when $\dim_{\mathbb{F}_2}(D^\mu)=40$ and $(\lfloor n/2\rfloor+1,\lfloor (n-1)/2\rfloor )$ when $\dim_{\mathbb{F}_2}(D^\mu)=2^{\lfloor (n-1)/2\rfloor}$ for $7\leq n\leq 9$, as desired.

Conversely, suppose that $p=2$ and $\mu=(\lfloor n/2\rfloor+1,\lfloor (n-1)/2\rfloor )$ for $7\leq n\leq 9$ or $(5,3,1)$. Note that $G=A_n$. If $\mu\neq (5,3,1)$, then $|V|<|G|$, so $G$ does not have a regular orbit on $V$, and if $\mu=(5,3,1)$, then we use {\sc Magma}  to check no orbit is regular. 

(ii)   If $G$ has no regular orbits on $V$, then $S_n\times\mathbb{F}_p^*$ has no regular orbits on $D^\mu$, so  $p=2$ and $\mu=(\lfloor n/2\rfloor+1,\lfloor (n-1)/2\rfloor )$ for $7\leq n\leq 10$ or $n=12$ by Proposition \ref{regular Sn}. In particular, $G=A_n$. Since $D^\mu\downarrow A_n$ is irreducible, the only possibilities for  $\mu$ are $(6,4)$ or $(7,5)$. If  $\mu=(7,5)$, then we use {\sc Magma}  to find a regular orbit of $G$ on $V$, a contradiction. Hence $\mu=(6,4)$, in which case $|V|<|G|$, so $G$ does not have a regular orbit on $V$.
\end{proof}

\subsection{Modules in $R_n(2)\setminus R_n(1)$}
\label{s: Rn(2)}

In this section, it is natural to work over an arbitrary field, and we obtain the following more general result for  modules in $R_n(2)\setminus R_n(1)$. 

\begin{prop}
\label{Rn(2) n>11}
Let $H$ be  $S_n$ or $A_n$ where $n\geq 5$. Let $G:= H\times A$ where $F$ is a field and $A$ is a finite subgroup of $F^*$. Let $V$ be a faithful irreducible $FH$-module  where  $V\leq D^\mu\downarrow H$ and $D^\mu\in R_n(2)\setminus R_n(1)$. If $n\geq 12$, then $G$ has a regular orbit on $V$.
\end{prop}

Proposition \ref{Rn(2) n>11} extends Step 5 of the proof of \cite[Theorem 6]{HalLieSei1992}, which  exhibits a regular orbit of $A_n$ on modules in $R_n(2)\setminus R_n(1)$ for $n>30$. Our methods of proof are similar. Although Proposition \ref{Rn(2) n>11} is trivial when $F$ is an infinite field by Lemma \ref{dag}, we include such $F$ here since Lemma \ref{dag} only guarantees the existence of a regular orbit, whereas our proof is constructive.

For modules in $R_n(2)\setminus R_n(1)$, we are  primarily concerned with the partitions $(n-2,2)$ and $(n-2,1,1)$. For these partitions, the modules $M^\mu$ and $S^\mu$ can be understood most readily using graphs. We  assume a  familiarity with basic terminology from graph theory throughout this section.

If $\mu=(n-2,2)$, then $M^\mu$ is the permutation module of $S_n$ on the set of unordered pairs from $\{1,\ldots,n\}$, so the set of simple undirected graphs on $n$ vertices with  edges weighted by field elements is isomorphic to $M^\mu$ if we identify each unordered pair $\{i,j\}$ with the edge whose ends are $i$ and $j$. With this viewpoint, the Specht module $S^\mu$ is spanned by the  alternating $4$-cycles, which are graphs of the form $\{i,j\}-\{j,k\}+\{k,l\}-\{l,i\}$ for  distinct $i,j,k,l\in \{1,\ldots,n\}$. Observe that the sum of $\{1,2\}-\{2,3\}+\{3,4\}-\{4,1\}$ and $ \{1,4\}-\{4,5\}+\{5,6\}-\{6,1\}$ is the alternating $6$-cycle $\{1,2\}-\{2,3\}+\{3,4\}-\{4,5\}+\{5,6\}-\{6,1\}$. Continuing in this way, we conclude that $S^\mu$ contains every alternating $2m$-cycle for $m\geq 2$.

Similarly, if $\mu=(n-2,1,1)$, then $M^\mu$ is the permutation module of $S_n$ on the set of ordered pairs from $\{1,\ldots,n\}$, so the set of simple directed graphs on $n$ vertices with  edges weighted by field elements is isomorphic to $M^\mu$ if we identify each ordered pair $(i,j)$ with the edge whose tail  is $i$ and head  is $j$.  With this viewpoint, the Specht module $S^\mu$ is spanned by the  directed $3$-cycles, which are graphs of the form $(i,j)-(j,i)+(j,k)-(k,j)+(k,i)-(i,k)$ for  distinct $i,j,k\in \{1,\ldots,n\}$. Observe that the sum of $(1,2)-(2,1)+(2,3)-(3,2)+(3,1)-(1,3)$ and $(1,3)-(3,1)+(3,4)-(4,3)+(4,1)-(1,4)$ is the directed $4$-cycle $(1,2)-(2,1)+(2,3)-(3,2)+(3,4)-(4,3)+(4,1)-(1,4)$. Continuing in this way, we conclude that $S^\mu$ contains every directed $m$-cycle for $m\geq 3$.

\begin{lemma}
\label{Rn(2) irr}
 Let $F$ be a field, and suppose that $n\geq 7$. If $V$ is an $FS_n$-module in $R_n(2)\setminus R_n(1)$, then $V\downarrow A_n$ is irreducible.
 \end{lemma}

\begin{proof}
For $n>30$, this is proved in Step 5 of the proof of \cite[Theorem 6]{HalLieSei1992}. We reproduce this proof here in order to deal with smaller $n$. Since the modules in $R_n(2)\setminus R_n(1)$ have the form $D^\mu$ or $D^\mu\otimes_F\sgn$ where $\mu$ is $(n-2,2)$ or $(n-2,1,1)$, and since $D^\mu\downarrow A_n\simeq D^\mu\otimes_F \sgn \downarrow A_n$, it suffices to assume that $V=D^\mu$ where $\mu$ is $(n-2,2)$ or $(n-2,1,1)$.

 Suppose for a contradiction that $D^\mu\downarrow A_n$ is not irreducible, and let $W$ be an irreducible $FA_n$-submodule of $D^\mu$. 
 Lemma \ref{index 2} implies that $D^\mu=W\oplus Wg$ where  $g=(12)\in S_n$.  Recall that $\dim_F([W,g])=\dim_F(W)-\dim_F(C_W(g))$. But $C_W(g)=0$ since $W\cap Wg=0$, and $ (n^2-5n+2)/2\leq \dim_F(D^\mu)$ by \cite[Appendix Table 1]{JamG1983},  
 so $ (n^2-5n+2)/4\leq \dim_F(D^\mu)/2=\dim_F([W,g])$.
 Moreover,  $\dim_F([W,g])\leq \dim_F([M^\mu,g])=\dim_F(M^\mu)-\dim_F(C^\mu)$ where $C^\mu:=C_{M^\mu}(g)$, and $\dim_F(M^\mu)$ is either $n(n-1)/2$ or $n(n-1)$ when $\mu$ is $(n-2,2)$ or $(n-2,1,1)$ respectively. Hence  $\dim_F(C^{(n-2,2)})\leq (n^2+3n-2)/4$ and $\dim_F(C^{(n-2,1,1)})\leq (3n^2+n-2)/4 $.
 
 Now we consider the dimension of $C^\mu$ and compare it to the upper bounds above to obtain a contradiction for all but the smallest $n$.   Suppose that $\mu=(n-2,2)$. 
 The graphs $\{1,2\}$,  $\{i,j\}$ and $\{1,i\}+\{2,i\}$ are fixed by $g$ for all $i,j\notin\{1,2\}$, and these form a linearly independent set in $M^\mu$, so $\dim_F(C^\mu)\geq 1+(n-2)(n-3)/2+(n-2)=(n^2-3n+4)/2$. But this is impossible unless $n=7$. Next suppose that $\mu=(n-2,1,1)$. The graphs $(i,j)$, $(1,2)+(2,1)$, $(1,i)+(2,i)$, and $(i,1)+(i,2)$ are fixed by $g$ for all $i,j\notin\{1,2\}$, and again these form a linearly independent set, so $\dim_F(C^\mu)\geq (n-2)(n-3)+1+(n-2)+(n-2)=n^2-3n+3$. But this is impossible unless $n\leq 11$.
  
 Hence either $n=7$ when $\mu=(n-2,2)$, or $7\leq n\leq 11$ when $\mu=(n-2,1,1)$. If $n=11$ and $\Char(F)=11$, then $\dim_{F}(D^\mu)=36$, and so $D^\mu\downarrow A_n$ is irreducible by \cite{BAtlas}. Otherwise, $D^\mu\downarrow A_n$ is irreducible by \cite{GAP4,SpinSym} (including the case $\Char(F)>n$ or $\Char(F)=0$).
\end{proof}
  
 Note that when $n=5$ or $6$, there are examples of $FS_n$-modules $V$ in $R_n(2)\setminus R_n(1)$ for which $V\downarrow A_n$ is not irreducible.

We will need the following technical result about  graphs. For a graph $\Gamma$, we denote the vertex set of $\Gamma$ by $V\Gamma$ and the edge set of $\Gamma$ by $E\Gamma$, and for $u\in V\Gamma$, we denote the valency of $u$ by $\Deg{u}$ and the set of vertices adjacent to $u$ by $\Gamma(u)$. We denote the complete graph with $\ell$ vertices by $K_\ell$ and the complete bipartite graph with parts of size $\ell$ and $\ell'$ by $K_{\ell,\ell'}$.
 
\begin{lemma} 
\label{graph}
Let $\Gamma$ be a finite simple undirected graph. Suppose that $|V\Gamma|\geq 12$ and $1\leq |E\Gamma|\leq 2|V|+8$, and suppose that the maximal valency of $\Gamma$ is at most $8$.
Either there exist distinct  $v_1,v_2,v_3,v_4\in V\Gamma$ such that  $\{v_1,v_2\}\in E\Gamma$ but  $\{v_2,v_3\},\{v_3,v_4\},\{v_4,v_1\}\notin E\Gamma$, or $|V\Gamma|=12$ and $\Gamma$ is one of $K_{4,8}$, $K_6\cup K_6$ or $K_5\cup K_7$.
\end{lemma}

\begin{proof}
Let $a\in V\Gamma$ have minimal non-zero valency. Note that if $u\in V\Gamma$ has valency 0, then since $|V\Gamma|\geq 12$ and $\Deg{a}\leq 8$, we may take  $v_4\in V\Gamma\setminus (\Gamma(a)\cup\{a,u\})$ along with $v_1=a$, $v_2\in \Gamma(a)$ and $v_3=u$. Thus we may assume that every vertex has non-zero valency. In particular, since $2|E\Gamma|=\sum_{v\in V\Gamma}\Deg{v}$, it follows that  $2|E\Gamma|\geq |V\Gamma|\Deg{a}$. Thus $\Deg{a}\leq 5$, or else $6|V\Gamma|\leq 2|E\Gamma|\leq 2(2|V\Gamma|+8)$, and so $|V\Gamma|\leq 8$, a contradiction.
Choose $b\in V\Gamma\setminus \Gamma(a)$ with maximal valency. Let $A:=\Gamma(a)\setminus \Gamma(b)$, let $B:=\Gamma(b)\setminus \Gamma(a)$, and let $C:= V\Gamma \setminus (\Gamma(a) \cup \Gamma(b)\cup \{a,b\})$.

Suppose first that $C\neq \varnothing$. If $A\neq \varnothing$, then let $v_1=a$, $v_3=b$ and choose $v_2\in A$ and $v_4\in C$. By the symmetry of this argument, we may assume that $\Gamma(a)=\Gamma(b)$. Now $\Deg{a}=\Deg{b}$, but $a$ has minimal valency and $b$ has maximal valency in $V\Gamma\setminus \Gamma(a)$, so every vertex of $C$ has the same valency as $a$ and $b$. If there is an edge whose ends are both in $C$, then we may take  $v_1$ and $v_2$ to be the ends of this edge along with $v_3=b$ and $v_4=a$. Otherwise, every vertex of $C$ is adjacent to every vertex of $\Gamma(a)$.
Now every vertex of $\Gamma(a)$ has valency at least $|C|+2$, so $|C|\leq 6$, but $\Deg{a}\leq 5$, so  $|V\Gamma|=2+\Deg{a}+|C|\leq 13$. If $|V\Gamma|=13$, then  $\Gamma$ must contain a subgraph isomorphic to $K_{5,8}$, but $K_{5,8}$ has 40 edges, so $\Gamma$ has at least 40 edges,  contradicting our assumption that $|E\Gamma|\leq 34$. Similarly, if $|V\Gamma|=12$, then $\Gamma$ must contain a subgraph isomorphic to $K_{5,7}$ or $K_{4,8}$, but $K_{5,7}$ has 35 edges, $K_{4,8}$ has 32 edges, and  $|E\Gamma|\leq 32$, so $\Gamma$ must  be $K_{4,8}$. 

Thus we may assume  that $C$ is empty. Note that $|B|=|V\Gamma|-\Deg{a}-2\geq 12-5-2=5$. 
This implies that $|\Gamma(a)\cap \Gamma(b)|\leq 3$, so $|A|+|B|=|V\Gamma|-|\Gamma(a)\cap \Gamma(b)|-2\geq 12-3-2=7$.  Moreover, $A\neq \varnothing$ since  $\Deg{b}\leq 8$ and $|V\Gamma|\geq 12$.  If there is an edge that has one end in $A$ and the other in $B$, then we may take these ends to be $v_1$ and $v_2$ respectively along with $v_3=a$ and $v_4=b$, so we assume that there is no such edge. Suppose further that there exists $u\in \Gamma(a)\cap \Gamma(b)$. The vertex $u$ cannot be adjacent to every vertex of $A$ and $B$ or else $\Deg{u}\geq 9$. If $u$ is not adjacent to some vertex of $A$, then we take $v_1=a$, $v_2=u$, $v_3\in A\setminus \Gamma(u)$, and $v_4\in B$. Thus by symmetry we may assume that $\Gamma(a)\cap \Gamma(b)$ is empty, so that $\Gamma$ consists of exactly two connected components. 

If the component containing $a$ is not  complete, then we may choose distinct non-adjacent vertices $v_1,v_4\in A$  and take $v_2=a$ and $v_3=b$. By symmetry, we may assume that the components of $\Gamma$ are $K_{|A|+1}$ and $K_{|B|+1}$. Since $|A|\leq 5$ and $|B|\leq 8$,   the initial assumptions on $|V\Gamma|$ and $|E\Gamma|$ imply that $|V\Gamma|=12$, so $\Gamma$ is $K_6\cup K_6$ or $K_5\cup K_7$.
\end{proof}

For $s\in S^\mu$,  the \textit{underlying  graph} of $s$ is  either the  graph  $s$ with weights removed when $\mu=(n-2,2)$, or the graph $s$ with weights, direction and multiple edges removed when $\mu=(n-2,1,1)$. Thus the underlying graph of $s\in S^\mu$ is always a finite simple undirected graph. Recall that $ D^\mu=S^\mu/(S^\mu\cap S^{\mu\perp})$.

\begin{lemma}
\label{reg} 
Let $\mu$ be  $(n-2,2)$ or $(n-2,1,1)$, and suppose that $n\geq 12$. Let $F$ be a field for which $\mu$ is $\Char(F)$-regular, and let $A$ be a finite subgroup of $F^*$. If there exists $s\in S^\mu$ whose underlying graph has trivial automorphism group, maximal valency at most $4$,   and  at most $n+4$ edges when $n\geq 13$ or at most $14$ edges when $n=12$, then $S_n\times A$ has a regular orbit on  $D^\mu$ and $D^\mu\otimes_F \sgn$. 
\end{lemma}

\begin{proof}
We claim  it suffices to prove that  $s-\lambda sg\notin S^{\mu\perp}$ for all $1\neq g\in S_n$ and $\lambda\in F^*$. Suppose that this occurs.
Then $s\not\in S^{\mu\perp}$ since $S^{\mu\perp}$ is an $FS_n$-submodule of $M^\mu$. If $(s+S^\mu\cap S^{\mu\perp})g\lambda =s+S^\mu\cap S^{\mu\perp}$ for some $g\in S_n$ and $\lambda\in A$, then $s-\lambda sg\in S^{\mu\perp}$, so $g=1$. But $s\not\in S^{\mu\perp}$, so $\lambda=1$. Hence $S_n\times A$ has a regular orbit on $D^\mu$. Moreover, if $(s+S^\mu\cap S^{\mu\perp}\otimes 1)g\lambda =s+S^\mu\cap S^{\mu\perp}\otimes 1$ for some $g\in S_n$ and $\lambda\in A$, then either $g\in A_n$ and $s-\lambda sg\in S^{\mu\perp}$, or $g\in S_n\setminus A_n$ and $s+\lambda sg\in S^{\mu\perp}$. If the latter  holds, then $g=1$, but this is ridiculous since $g\notin A_n$, so the former  holds. Again $g=1$, so $\lambda=1$. Hence $S_n\times A$ has a regular orbit on $D^\mu\otimes_F \sgn$, and the claim is proved.

Fix $1\neq g\in S_n$ and $\lambda\in F^*$. Now $s-\lambda sg\neq 0$, or else $g$ is a non-trivial automorphism of the underlying graph of $s$. Moreover, the underlying graph $\Gamma$ of $s-\lambda sg$ has at most $2n+8$ edges when $n\geq 13$ or at most $28$ edges when $n=12$,  and its vertices have valency at most 8. Note that if $n=12$, then $\Gamma$ cannot be $K_{4,8}$, $K_6\cup K_6$ or $K_5\cup K_8$, as these graphs have too many edges. Hence Lemma \ref{graph} implies that there exist distinct vertices $i,j,k,l$ such that $\{i,j\}$ is an edge of $\Gamma$ but $\{j,k\}$, $\{k,l\}$ and $\{l,i\}$ are not edges of $\Gamma$.  Let 
$$
s':=\left \{ 
\begin{array}{l l} 
\{i,j\}-\{j,k\}+\{k,l\}-\{l,i\} & \mbox{if}\ \mu=(n-2,2), \\  
(i,j)-(j,i)+(j,k)-(k,j)  & \\
\ \ \ \ +(k,l)-(l,k)+(l,i)-(i,l) &  \mbox{if}\ \mu=(n-2,1,1), \\
\end{array} 
\right .
$$
so that $s'\in S^\mu$. We claim that $<\!s-\lambda sg,s'\! >\neq 0$, in which case $s-\lambda sg\notin S^{\mu\perp}$, as desired. Certainly this holds if $\mu=(n-2,2)$ since $<\!s-\lambda sg,s'\! >$ is the weight of the edge $\{i,j\}$ in $s-\lambda sg$, so we assume that $\mu=(n-2,1,1)$. Observe that  $(u,v)$ is an edge of  $t\in S^\mu$ if and only if $(v,u)$ is an edge of $t$. Also, if $(u,v)$ has weight $\delta$ in $t$, then $(v,u)$ has weight $-\delta$ in $t$. Let $\delta$ be the weight of $(i,j)$ in $s-\lambda sg$. Now $<\!s-\lambda sg,s'\! >=<\!\delta(i,j)-\delta(j,i),(i,j)-(j,i)\! >=2\delta\neq 0$ since $\mu$ is $\Char(F)$-regular.
\end{proof}

\begin{proof}[Proof of Proposition $\ref{Rn(2) n>11}$]
By Lemma \ref{Rn(2) irr}, it suffices to show that $S_n\times A$ has a regular orbit on $V$, where $V$ is $D^\mu$ or $D^\mu\otimes_{F}\sgn$ and $\mu$ is $(n-2,2)$ or $(n-2,1,1)$.

Suppose first  that $n\geq 13$. Let $m:=2\lfloor n/2\rfloor$. If $\mu=(n-2,2)$, then define
$$
\begin{array}{rl}
s_1 & :=\{1,2\}-\{2,4\}+\{4,5\}-\{5,1\},\\
s_2 & :=\{2,3\}-\{3,4\}+\{4,6\}-\{6,2\},\\
s_3 & :=\{5,6\}-\{6,7\}+\cdots +\{m\!-\!1,m\}-\{m,5\},\\ 
\end{array}
$$
and if $\mu=(n-2,1,1)$, then define $s_1$, $s_2$ and $s_3$ by replacing each weighted edge  $\pm\{i,j\}$  above by $(i,j)-(j,i)$. Note that $s_1$, $s_2$ and $s_3$ are in $S^\mu$ in either case. Let $s:=s_1+s_2+s_3$. Now the underlying graph of $s$ has $m+4$ edges and maximal valency 4. Moreover, it is routine to verify that the underlying graph of $s$ has a trivial automorphism group. Thus  $S_n\times A$ has a regular orbit on $D^\mu$ and $D^\mu\otimes_F \sgn$ for $n\geq 13$  by Lemma \ref{reg}. 

Now suppose that $n=12$ and $|F|\neq 2$.
We may choose non-zero elements $\lambda_1$ and $\lambda_2$ of $F$ such that $\lambda_1+\lambda_2\neq 0$. For $\mu=(n-2,2)$, define
$$
\begin{array}{rl}
s_1 & :=\lambda_1(\{1,2\}-\{2,3\}+\{3,4\}-\{4,1\}),\\ 
s_2 & :=\lambda_2(\{3,4\}-\{4,5\}+\{5,6\}-\{6,7\}+\{7,8\}-\{8,3\}),\\
s_3 & :=\lambda_1(\{7,8\}-\{8,9\}+\{9,10\}-\{10,11\}+\{11,12\}-\{12,7\}),\\
\end{array}
$$
and for $\mu=(n-2,1,1)$, define $s_1$, $s_2$ and $s_3$ by replacing each weighted edge  $\pm\lambda_k\{i,j\}$  by $\lambda_k(i,j)-\lambda_k(j,i)$. Now  $s:=s_1+s_2+s_3\in S^\mu$ and the underlying graph of $s$ has $14$ edges, maximal valency 3 and trivial automorphism group, so we are done by Lemma \ref{reg}.

Lastly, if $n=12$ and $|F|=2$, then $S_n$ has a regular orbit on $V$ by Proposition \ref{regular Sn}(ii).
\end{proof}

\subsection{Modules in $R_n(1)$}
\label{s: Rn(1)}

Now we  find the only infinite class of faithful irreducible modules on which $S_n$ has no regular orbits. Neither module in $R_n(0)$ is faithful for $n\geq 5$, so we are only concerned with modules in $R_n(1)\setminus R_n(0)$. In particular, we are primarily concerned with the partition $(n-1,1)$. 

Let $F$ be a field and $\mu=(n-1,1)$. The $FS_n$-module $M^{\mu}$ is  the permutation module $F^n$ where $S_n$ acts on $F^n$ by permuting the coordinates.  The \textit{deleted permutation module}
 $S^\mu=\left \{(a_1,\ldots,a_n)\in F^n:\sum_{i=1}^na_i=0\right\}$  and has dimension $n-1$. Moreover, $S^{\mu\perp}=\{(a,\ldots,a)\in F^n\}$, and clearly $S^\mu\cap S^{\mu\perp}$ is either 0 when $p\nmid n$, or $S^{\mu\perp}$ when $p\mid n$, so the \textit{fully deleted permutation module} $D^\mu$ has dimension $n-1$ if $p\nmid n$ and dimension $n-2$ if $p\mid n$. 
Note that $D^\mu\downarrow A_n$ is irreducible  for $n\geq 5$.

The regular orbits of $S_n\times \mathbb{F}_q^*$ on $S^{(n-1,1)}$  were determined by Gluck \cite{Glu2007}, and also   Schmid \cite{SchP2007} for $\Char(\mathbb{F}_q)>n$. We now determine the regular orbits on $D^{(n-1,1)}$.

\begin{prop}
\label{Rn(1)}
 Let $V$ be the $\mathbb{F}_pS_n$-module $D^{(n-1,1)}$ where $n\geq 5$ and $p$ is a prime such that $p\leq n$.
\begin{itemize}
\item[(i)] If $S_n\leq G\leq S_n\times\mathbb{F}_p^*$, then $G$ does not have a regular orbit on  $V$ or $V\otimes_{\mathbb{F}_p}\sgn$.
\item[(ii)] If $A_n\leq G\leq A_n\times \mathbb{F}_p^*$, then $G$ has a regular orbit on $V$ if and only if $G=A_n$ and $p=n-1$.
\end{itemize}
\end{prop}

\begin{proof}
Let $\mu:=(n-1,1)$ and   $W:=S^\mu\cap S^{\mu\perp}$.  Let $H\leq G\leq H\times\mathbb{F}_p^*$ where $H=S_n$ or $A_n$.
We prove (i) and (ii) simultaneously by considering the various possibilities for $p$ in relation to $n$. 

Suppose that $p\leq n-2$. Clearly every $n$-tuple of elements from $\mathbb{F}_p$ must contain either three repeated entries or two pairs of repeated entries, so every element of $V$ is fixed by some non-trivial element of $A_n$. But if $A_n$ has no regular orbits on $V$, then $G$ has no regular orbits  on  $V$ or $V\otimes_{\mathbb{F}_p} \sgn$, so this case is complete.

Suppose that $p=n$. Again, it suffices to prove that  $A_n$ has no regular orbits on $V$. Let $v+W\in V$. Note that if $v$ has exactly one pair of repeated entries, then there is exactly one $b\in \mathbb{F}_p$ that does not appear in $v$, but the sum of the  elements of $\mathbb{F}_p$ vanishes, as does the sum of the coordinates of $v$, so $b$ must be the repeated entry, a contradiction. Moreover, if $v$ has at least two pairs of repeated entries or a triple of repeated entries, then $v+W$ is certainly fixed by a non-trivial element of $A_n$. Hence we may  assume that $v$ is of the form $(v_1,\ldots,v_p)$ where $v_i\neq v_j$ for all $i\neq j$. Let $g\in S_n$ be the permutation for which $vg=(v_1+1,\ldots,v_p+1)$. Now $g$ has no fixed points and fixes $v+W$. Moreover, $g$ must be a $p$-cycle, for  if $(i_1\cdots i_k)$ is a cycle of $g$ for some $k\in \{2,\ldots, p\}$, then  $v_{i_k}=v_{i_1}+1$ and $v_{i_j}=v_{i_{j+1}}+1$ for all $j\in\{1,\ldots, k-1\}$, and it follows that $v_{i_1}=v_{i_1}+k$. Thus $k=p$ and $g\in A_n$, as desired.

Suppose that $p=n-1$. Then $V=S^\mu$. 
First we claim that if $1\neq \lambda\in \mathbb{F}_p^*$ and $v$ is an element of $V$ with exactly one pair of repeated entries, then there exists $1\neq g\in A_n$ such that $vg\lambda=v$. Write $v=(v_1,\ldots, v_n)$, and let $i$ and $j$ be the unique pair of distinct indices for which $v_i=v_j$.   Since the sum of the elements in $\mathbb{F}_p$ is zero, it follows that $v_i=0$. Thus the non-zero entries of $v$ are the $p-1$ distinct elements of $\mathbb{F}_p^*$, and so the set of entries of $v$ contains a transversal $T$ for the cosets of $\langle \lambda\rangle$ in $\mathbb{F}_p^*$.
Let $m$ be the order of $\lambda$ in $\mathbb{F}_p^*$.
As in the proof of \cite[Lemma 2]{Glu2007}, each coset of $\langle \lambda\rangle$ has the form $\{v_{i_0},\ldots,v_{i_{m-1}}\}$ for some $v_{i_0}\in T$, where $v_{i_j}=\lambda^jv_{i_0}$ for $1\leq j\leq m-1$.  Define $g\in A_n$ to be the product of the disjoint cycles $(i_0,\ldots,i_{m-1})$ and  $(i\ j)$ if needed. 
Now $g\neq 1$ and $vg\lambda=v$, as claimed.

If $G=A_n$, then $(1,2,\ldots,p-1,0,0)$ lies in a regular orbit of $G$, so we may assume that $G\neq A_n$. We claim that $G$ does not have a regular orbit on $V$ or $V\otimes_{\mathbb{F}_p}\sgn$. Let $0\neq v\in V$. If $v$ has a triple of repeated entries or two pairs of repeated entries, then some $g\in A_n$ fixes both $v$ and $v\otimes 1$, so we may assume that $v$ has exactly one pair of repeated entries, say with indices $i$ and $j$. Suppose that $H=S_n$.
Some element of $S_n$ will fix $v$, so $G$ has no regular orbits on $V$. Moreover, by the claim there exists $1\neq g\in A_n$ such that $vg=-v$, so $g(i\ j)$ fixes $v\otimes 1$, and we conclude that $G$ has no regular orbits on $V\otimes_{\mathbb{F}_p}\sgn$. Thus $H=A_n$.  Since $A_n<G$,  there exists $1\neq \lambda\in G\cap\mathbb{F}_p^*$. Now there exists $1\neq g\in A_n$ such that $vg\lambda=v$ by the claim, and $g\lambda\in G$, so $G$ has no regular orbits on $V$. 
\end{proof}

\subsection{Proof of  Theorem \ref{regular}}
\label{s: sym proof}

Let $d:=\dim_{\mathbb{F}_p}(V)$. There exists a non-negative integer $m$ for which $D^\mu\in R_n(m)$. Since $V$ is faithful as an $\mathbb{F}_pH$-module,  $D^\mu$ is faithful as an $\mathbb{F}_pS_n$-module, so $D^\mu\notin R_n(0)$. If $D^\mu\in R_n(1)\setminus R_n(0)$, then  (i) holds by Proposition \ref{Rn(1)}, so we assume that $D^\mu\notin R_n(1)$, in which case the condition of (ii) holds.

Suppose that $n\geq 7$. Note that for  $\mu$ listed in Table \ref{tab: Sn ex},  $\dim_{\mathbb{F}_p}(V)$ is as listed by   \cite{JamG1978,BAtlas}. If $D^\mu\not\in R_n(2)$, then we are done  by Propositions \ref{regular Sn}(i) and \ref{regular An}, so we assume  that $D^\mu\in R_n(2)\setminus R_n(1)$, and if $H=S_n$, then we are done by Propositions \ref{regular Sn}(ii) and  \ref{Rn(2) n>11}. Hence we assume that $H=A_n$, in which case we are done by Proposition \ref{Rn(2) n>11} unless  $n\leq 11$. Recall that $V\downarrow A_n=D^\mu$ by Lemma \ref{Rn(2) irr}, and suppose that $G$ does not have a regular orbit on $V$.
Then $S_n\times \mathbb{F}_p^*$ does not have a regular orbit on $D^\mu$, so Proposition \ref{regular Sn}(ii) implies that $p=2$ and $\mu=(n-2,2)$ for $7\leq n\leq 8$. If $\mu=(5,2)$, then  $G$ has a regular orbit on $V$ by  {\sc Magma}, so $\mu=(6,2)$, in which case  $G$ does not have a regular orbit on $V$ since  $|V|<|G|$.

 Thus  $n=5$ or 6. Using \cite{GAP4,SpinSym}, we determine the possibilities for $\mu$ and $d$. If $(n,p,\mu,G,d)$ is not listed in Table \ref{tab: Sn ex}, then  $G$ has a regular orbit on $V$ by {\sc Magma}, so we may assume that $(n,p,\mu,G,d)$ is listed in Table \ref{tab: Sn ex}.  If $\mu$ is $(3,2)$ or $(4,2)$, then $|V|<|G|$, so $G$ does not have a regular orbit on $V$. Otherwise, we determine that $G$ has no regular orbits on $V$ using {\sc Magma}. \hfill $\Box$

\section{Covering groups}
\label{s: cover props}

 Recall  that the proper covering groups of $S_n$ and $A_n$ are $2.S_n^+$ and $2.S_n^-$ for $n\geq 5$,  $2.A_n$ for $n\geq 5$, and  $3.A_n$ and $6.A_n$ for $n=6$ or $7$.  
We focus on the double covers of $S_n$ and $A_n$ for most of this section, as $3.A_n$ and $6.A_n$ will be dealt with computationally. 
For $n\geq 5$,
\begin{align*}
 2.S_n^+&:=\langle z, s_1,\ldots,s_{n-1}:z^2=1,s_i^2=(s_is_{i+1})^3=1,zs_i=s_iz,(s_is_j)^2=z\ \mbox{if}\ |i-j|>1\rangle,\\
2.S_n^-&:=\langle z, t_1,\ldots,t_{n-1}:z^2=1,t_i^2=(t_it_{i+1})^3=z,zt_i=t_iz,(t_it_j)^2=z\ \mbox{if}\ |i-j|>1\rangle.
\end{align*}
The centre of each group is $\{1,z\}$, and they are isomorphic precisely when $n=6$, in which case we write $2.S_6$.
We also write $2.S_n^{\varepsilon}$ when no distinction between the two  covers needs to be made. The cover  $2.A_n$ is  the derived subgroup of $2.S_n^\varepsilon$ and  has centre $\{1,z\}$.

Let $G$ be $2.S_n^\varepsilon$ or $2.A_n$ where $n\geq 5$, and let $z$ be the unique central involution of $G$. Let $F$ be a field, and let $V$ be an irreducible $FG$-module. Now $z$ must act as $1$ or $-1$ on $V$. 
(Indeed, this is the case for a central involution in any finite group.) 
Since every non-trivial normal subgroup of $G$ contains  $z$, it follows that $V$ is faithful precisely when $z$ acts as $-1$. 
 In particular,  $G$ has no faithful irreducible representation over a field of characteristic $2$.
 In this section, we prove the following theorem.

 \begin{thm}
\label{regular double}
Let $H$ be a proper covering group of  $S_n$ or $A_n$ where $n\geq 5$.
Let $G$ be a group for which  $H\leq G\leq H\circ \mathbb{F}_p^*$ where $p$ is a prime and $p\leq n$. Let $V$ be a faithful irreducible $\mathbb{F}_pH$-module. Let $d:=\dim_{\mathbb{F}_p}(V)$.
 The group   $G$ has a regular orbit on $V$ if and only if  $(n,p,G,d)$  is not listed in Table $\ref{tab: total ex}$. 
\end{thm}

Let $G$ be $2.S_n^\varepsilon$ or $2.A_n$ where $n\geq 5$.  
We will be primarily interested in the so-called basic spin modules of $G$, for these have minimal dimension among the faithful irreducible modules by \cite{KleTie2004} (cf.\ Theorem \ref{KT}). In fact,  non-basic spin modules have such large dimension that they almost always have regular orbits. Indeed, there is only one non-basic spin module listed in Table \ref{tab: total ex}; it arises for $n=p=5$  when $G=2.A_5\circ\mathbb{F}_5^*$ and has dimension 4.

Over the complex numbers,  the irreducible  representations  of $G$ can be indexed by certain partitions of $n$ \cite{Sch1911}, and the complex basic spin modules 
 of $G$ are  those representations corresponding to the partition $(n)$.  For an algebraically closed field  of positive characteristic $p$, a \textit{basic spin module} is a composition factor of the reduction modulo $p$ of a complex basic spin module, and by \cite{Wal1979}, this reduction is irreducible except when $p\mid n$ and either $n$ is odd for $G=2.S_n^\varepsilon$,  or $n$ is even for $G=2.A_n$. Moreover,   there are at most two basic spin modules, and when there are two, they are either associates or conjugates. Lastly, every basic spin module of $2.A_n$ arises as a submodule of a basic spin module of $2.S_n^\varepsilon$.

Let $p$ be a prime. As in  \cite{KleTie2004}, define 
   $$
\delta(G):= \left \{
\begin{array}{l l}
 2^{\left\lfloor \tfrac{1}{2}\left( n-1-\kappa(p,n) \right)\right\rfloor} & \mbox{if}\ G=2.S_n^\varepsilon \\
  2^{\left\lfloor \tfrac{1}{2}\left( n-2-\kappa(p,n) \right)\right\rfloor} & \mbox{if}\ G=2.A_n ,\\
\end{array}
\right.
$$
where $\kappa(p,n):=1$ if $p\mid n$ and $0$ otherwise.
Now $\delta(G)$ is the dimension of a basic spin module of $G$ over a splitting field of characteristic $p$ 
 \cite{Wal1979}. Moreover, Kleshchev and Tiep  \cite{KleTie2004}  provide a lower bound for the dimensions of faithful irreducible representations of $G$ in terms of $\delta(G)$, which we now state.

\begin{thm}
\label{KT}
Let $G$ be $2.S_n^\varepsilon$ or $2.A_n$ where $n\geq 8$, and let $F$ be an algebraically closed field of positive characteristic. If $V$ is a faithful irreducible $FG$-module for which $\dim_F{V}<2\delta(G)$, then $V$ is a basic spin module and $\dim_F{V}=\delta(G)$.
\end{thm}

 Theorem \ref{KT} can be applied to any finite  field in the following way. Let $V$ be a faithful irreducible $FG$-module where $F$ is a finite field and $G$ is $2.S_n^\varepsilon$ or $2.A_n$. Let $k:=\End_{FG}(V)$. Recall that $V$ is a faithful absolutely irreducible $kG$-module. If  $V$ is the realisation over $k$ of a basic spin module of $G$, then we also refer to $V$ as a basic spin module.  For $n\geq 8$, Theorem \ref{KT} implies that either $\dim_k(V)=\delta(G)$, in which case $V$ is a basic spin module, or
   $2\delta(G)\leq \dim_k(V)\leq \dim_{F}(V)$.

Unlike $S_n$, not every field is a splitting field for $2.S_n^\varepsilon$. However, every field containing $\mathbb{F}_{p^2}$ for $p$ an odd prime is a splitting field for $2.S_n^\varepsilon$ and $2.A_n$ (cf.   \cite[Corollary 5.1.5]{Maa2011}). Note that  there are instances where $\mathbb{F}_p$ is a splitting field for $2.A_n$ but not for $2.S_n^\varepsilon$.

The Brauer character tables of $2.S_n^+$ and $2.A_n$ for $p\leq n$ and $5\leq n\leq 12$ may be found in \cite{BAtlas} and also in \cite{GAP4} for $p\leq n$ and $5\leq n\leq 13$. The Brauer character tables of $2.S_n^-$ for $5\leq n\leq 18$ and $p\in \{3,5,7\}$ may be found in {\sf GAP} \cite{GAP4} via the  {\sc SpinSym} package \cite{SpinSym}.    We can convert the character table of one double cover to that of the other using {\sf  GAP}. When the reduction modulo $p$ of an ordinary irreducible representation of  $2.S_n^\varepsilon$ or $2.A_n$ is irreducible, the Brauer character is the ordinary character restricted to $p$-regular elements and  can therefore be accessed using the generic character tables  in {\sf GAP}.
 
We begin with a reduction for almost quasisimple groups $G$ with $F^*(G)'\simeq 2.A_n$.

\begin{lemma}
\label{2An reduction}
Let $G$ be an almost quasisimple group where   $N:=F^*(G)'\simeq 2.A_n$ and $n\geq 8$. Let $F$ be a finite field. Let $V$ be a faithful irreducible $FG$-module, $k:=\End_{FG}(V)$ and  $q:=|k|$. Let $W$ be an irreducible $kN$-submodule of $V$.
If $G$ has no regular orbits on $V$, then $n\leq 20$, and if $\Char(F)\leq n$, then the following hold.
\begin{itemize}
\item[(i)] If $n\geq 13$, then $W$ is a basic spin module and $(n,q)$ is listed in Table $\ref{tab: double big}$. If there is no $*$ next to $q$, then $V\downarrow N=W$ and $W$ is an  absolutely irreducible $kN$-module. 
\begin{table}[!ht]
\centering
\begin{tabular}{ c  c }
\hline
$n$ & $q$   \\
\hline 
$13$ & $3^*, 5^*, 7^*, 9^*, 11^*, 13^*, 25$\\
& $ 27, 49, 81, 121,  169, 243$  \\

$14$ & $3^*, 5, 7^*, 9, 11, 13, 49$ \\

 $15$ & $3^*,  5^*,  9, 11, 13, 25,  27$\\

$16$ & $3, 5, 7$ \\
 
 $ 17$ & $3^*,  7, 9, 11$ \\
 
  $18$ & $3^*, 9$\\

 $19$ & $3$  \\

 $20$ & $5$ \\
\hline
\end{tabular}
\caption{Possible $q$ when $n\geq 13$}
\label{tab: double big}
\end{table}
\item[(ii)] If $n\leq 12$ and $W$ is not a basic spin module, then $V\downarrow N=W$ and $W$  is an absolutely irreducible $kN$-module where $(n,q,\dim_k(W))$ is listed in Table $\ref{tab: double small}$.
\begin{table}[!ht]
\centering
\begin{tabular}{ c c c c }
\hline
$n$ &  $q$ & $\dim_k(W)$   \\
\hline 
$8$ &  $9$ & $24$\\
 &   $7$ & $16$ \\
 $9$ & $3$ & $48$ \\
 $10$ & $3$, $5$ & $48$ \\
 $11$ & $5$ &  $56$ \\
  \hline
\end{tabular}
\caption{Possible $q$ and $\dim_k(W)$ when $n\leq 12$ and $W$ is non-basic}
\label{tab: double small}
\end{table}
\end{itemize}
\end{lemma}

\begin{proof}
Suppose that $G$ has no regular orbits on $V$.
Let $p:=\Char(F)$.
Since $V$ is a faithful absolutely irreducible  $kG$-module, Lemma \ref{scalar} implies that $Z(G)\leq k^*$, and so $|Z(G)|\leq q-1$. Let
$$g(q,n):=\max{\{(n-1) \log_q{(n(n-1)(q-1))}, \tfrac{n}{2}\log_q{(2n!(q-1))}\}}.$$
Now equation (\ref{better bound}) of Lemma \ref{bounds} implies that $\dim_{k}(V)\leq  \lfloor g(q,n) \rfloor$.   Note that if $n$ is fixed, then $g(q,n)$ is a decreasing function in $q$ by Lemma \ref{decreasing}.

 Since $Z(N)\leq Z(G)$ by Lemma \ref{quasi}, the central involution  of $N$ must act as $-1$ on $V$. Thus $p\neq 2$ and  $W$ is a faithful $kN$-module, so  
 $\delta(N)\leq\dim_k(W)$ by Theorem \ref{KT}. In particular, $\delta(N)\leq \lfloor g(q,n)\rfloor $. If $n\geq 21$, then  $\lfloor g(q,n)\rfloor\leq\lfloor  g(3,n)\rfloor< 2^{\lfloor (n-3)/2 \rfloor}\leq \delta(N)$, a contradiction. Thus $n\leq 20$, as claimed.

We assume for the remainder of the proof that $p\leq n$. Let $E:=\End_{kN}(W)$. Suppose that $n\geq 13$. First we claim that either $(n,q)$ is listed in Table \ref{tab: double big} or  $(n,q)\in P$ where $$P:=\{(13,125),(13,343),(14,343),(15,7),(17,5)\}.$$ If $q\geq 7$ and $n\geq 19$, then  $\lfloor g(q,n)\rfloor \leq \lfloor g(7,n)\rfloor < 2^{\lfloor (n-3)/2 \rfloor}\leq \delta(N)$, a contradiction. Hence if $q\geq 7$, then $n\leq 18$. Similarly, if $q\geq 17$, then $n\leq 16$; if $q\geq 49$, then either $n=16$ or $n\leq 14$; if $q\geq 121$, then $n\leq 14$; and if $q\geq 625$, then either $n=14$ or $n\leq 12$. If $n=14$ where $q\geq 25$ and $p\neq 7$, then $\lfloor g(q,14)\rfloor\leq \lfloor g(25,14)\rfloor<64=\delta(N)$, a contradiction, and if $n=14$ and $ 7^4\mid q$, then $\lfloor g(q,14)\rfloor \leq\lfloor  g(7^4,14)\rfloor<32=\delta(N)$, a contradiction. Similarly, if $n=16$, then $\kappa(p,n)=0$, so we obtain a contradiction for $q\geq 9$. In fact, if $(n,q)$ is one of $(20,3)$, $(19,5)$, $(17,13)$, or $(18,q)$ where $q\in \{5,7,11,13\}$, then $\kappa(p,n)=0$, and we obtain contradictions. The claim follows. 

Let $Q$ be the set of $(n,q)$  listed in Table \ref{tab: double big} with an adjacent *.
By the claim,   $2\delta(N)\leq \lfloor g(q,n)\rfloor$ if and only if $(n,q)\in Q$. This has several consequences. 

Firstly,  $W$ is a basic spin module, or else Theorem \ref{KT} implies that $2\delta(N)\leq\dim_{E}(W)\leq \dim_k(W)\leq \lfloor g(q,n)\rfloor$, and so  $(n,q)\in Q$, but for such $(n,q)$, there is no faithful irreducible $EN$-module whose dimension lies between $2\delta(N)$ and $\lfloor g(q,n)\rfloor$ by \cite{GAP4, SpinSym}, a contradiction. 
Secondly, if $(n,q)\notin Q$, then $W$ is an absolutely irreducible $kN$-module, for if not, then $\dim_k(W)\geq 
2\dim_E(W)=2\delta(N)$, so $2\delta(N)\leq \lfloor g(q,n)\rfloor $, a contradiction. Thirdly, $(n,q)$ is listed in Table \ref{tab: double big}, for if not, then $(n,q)\in P$ by the claim,
but this implies that $W$ is not an absolutely irreducible $kN$-module by \cite{GAP4,SpinSym}, so $(n,q)\in Q$, a contradiction.
Lastly, if $(n,q)\notin Q$, then   $V\downarrow N$ is irreducible, or else $2\dim_k(W)\leq \dim_k(V)$ by Lemma \ref{index 2},  so $2\delta(N) \leq \lfloor g(q,n) \rfloor$, a contradiction. Thus we have proved (i).

Henceforth, we may assume that $n\leq 12$ and $W$ is not a basic spin module. Now $2\delta(N)\leq \dim_E(W)\leq \lfloor g(p,n)\rfloor$. For each $(n,p)$, we use \cite{BAtlas,GAP4, SpinSym} to determine the possibilities for $\dim_E(W)$. Either these are the dimensions given in Table \ref{tab: double small}, or $n=8$, $p=5$ and $\dim_E(W)=24$. 
Since $\dim_E(W)\leq \lfloor g(q,n)\rfloor$, it follows that either $q$ is listed in Table \ref{tab: double small}, or $n=8$ and $q=3$ or $5$.
If $n=8$ and $p=3$ or $5$, then  no faithful irreducible $\mathbb{F}_{p^2}N$-module of dimension 24 can be realised over $\mathbb{F}_p$ \cite{BAtlas}, so $\dim_k(W)=48$ when $q=3$ or $5$, while $ \lfloor g(q,8)\rfloor<48$, a contradiction. Thus $(n,q,\dim_E(W))$ is listed in Table \ref{tab: double small}, in which case $W$ is an absolutely irreducible $kN$-module \cite{BAtlas}, so $E=k$. Lastly, if $V\downarrow N\neq W$, then   $2\dim_k(W)\leq \dim_k(V)\leq\lfloor g(q,n)\rfloor$ by Lemma \ref{index 2}, a contradiction. 
\end{proof}

Next we consider the double covers of the symmetric group.

\begin{prop}
\label{regular 2Sn}
Let $H:=2.S_n^\varepsilon$ where $n\geq 8$, and let $G$ be such that $H\leq G\leq H\circ \mathbb{F}_p^*$  where   $p$ is a prime and $p\leq n$. Let $V$ be a faithful irreducible $\mathbb{F}_pH$-module. 
\begin{itemize}
\item[(i)] If $\varepsilon=-$, then $G$ has no regular orbits on $V$ if and only if $\dim_{\mathbb{F}_p}(V)=\delta(H)$ and $(n,p)$ is one of $(8,3)$, $(8,5)$, $(9,3)$, or $(10,3)$.
\item[(ii)] If $\varepsilon=+$, then $G$ has no regular orbits on $V$ if and only if $\dim_{\mathbb{F}_p}(V)=\delta(H)$ and $(n,p)=(8,5)$ and $G=2.S_n^+\circ\mathbb{F}_p^*$.
\end{itemize}
\end{prop}

\begin{proof}
We will prove (i) and (ii) simultaneously. Suppose that $G$ does not have a regular orbit on $V$.  Lemma \ref{2An reduction} implies that $n\leq 20$. Let $k:=\End_{\mathbb{F}_pG}(V)$ and $q:=|k|$. Let $W$ be an irreducible $kN$-submodule of $V$ where $N:=F^*(G)'=2.A_n$. 

First we claim that $q$ is either $p$ or $p^2$. 
Let $\chi$ be the character of the  $kG$-module $V$ and $\mathbb{F}_p(\chi)$ the subfield of $k$ generated by $\mathbb{F}_p$ and the image of $\chi$. By \cite[Theorem VII.1.16]{BlaHup1981}, the $\mathbb{F}_pG$-module $V$ is a direct sum of $[k:\mathbb{F}_p(\chi)]$ irreducible $\mathbb{F}_pG$-modules,  so  $k=\mathbb{F}_p(\chi)$.
Since $\chi$ is also the character of the irreducible $\overline{k}G$-module $V\otimes_k\overline{k}$, where $\overline{k}$ denotes the algebraic closure of $k$, it follows from \cite[Theorem VII.2.6]{BlaHup1981} that $k$ is contained in the unique smallest splitting field for $G$ in $\overline{k}$. Since $\mathbb{F}_{p^2}$ is  a splitting field for $H$, the claim follows.

Suppose that $n\geq 12$. We claim that $\dim_{k}(V)=\delta(H)$ and that $(n,p,\varepsilon)$ is listed in Table \ref{tab: 2Sn}. By Lemma \ref{2An reduction},  $W$ is a basic spin module and $(n,q)$ is listed in Table \ref{tab: double big} for $n\geq 13$. 
 Suppose that either $n=12$ or $(n,q)$ is such that $(n,p)$ has an adjacent $*$ in Table \ref{tab: double big}.
 Then   $\dim_{k}(V)=\delta(H)$ by  \cite{BAtlas,GAP4,SpinSym}, and $\dim_{\mathbb{F}_p}(V)=64$
when $(n,p,\epsilon)=(12,11,-)$,  but equation (\ref{even better bound}) of Lemma \ref{bounds} implies that $\dim_{\mathbb{F}_p}(V)\leq  57$, a contradiction. Hence  the claim holds in this case.

\begin{table}[!h]
\centering
\begin{tabular}{ c c c c c c c c c c }
\hline
$n$ & 12 & 12 & 13 & 14 & 14 & 15 & 16 & 16 & 17, 18   \\
$p$ & 3, 5, 7 & 11 &  $p\leq n$ & 3, 7 & 11 & 3, 5 & 3 & 7 & 3 \\
$\varepsilon$ & $\pm$ & $+$ & $\pm$ & $\pm$ & $-$ & $\pm$ & $+$ & $-$ & $\pm$ \\
\hline
\end{tabular}
\caption{Possible $p$ and $\varepsilon$ when $n\geq 12$}
\label{tab: 2Sn}
\end{table}
 
We may therefore assume that $(n,q)$ is such that $(n,p)$ has no adjacent $*$ in Table \ref{tab: double big}. Now $q=p$,  $W=V\downarrow N$  and $W$ is a faithful absolutely irreducible $kN$-module by Lemma \ref{2An reduction}. Thus $\dim_k(V)=\dim_k(W)=\delta(N)$. But if either $n$ is even and $p\mid n$, or $n$ is odd and $p\nmid n$, then $ \delta(H)=2\delta(N)$, and so $\dim_k(V)<\delta(H)$, contradicting Theorem \ref{KT}. 
 
 We conclude that either $n$ is even and $p\nmid n$, or $n$ is odd and $p\mid n$. Now $\dim_k(V)=\dim_k(W)=\delta(N)=\delta(H)$.
 If $n=14$ and $(p,\varepsilon)$ is one of $ (5,\pm)$, $(11,+)$ or $(13,\pm)$, or $n=16$ and $(p,\varepsilon)$ is one of $(3,-)$, $(5,\pm)$ or $(7,+)$, then $k=\mathbb{F}_{p^2}$ by \cite{GAP4,SpinSym},  a contradiction.

 Thus $(n,p,\varepsilon)$ is listed in Table \ref{tab: 2Sn} and $\dim_{k}(V)=\delta(H)$.
Using {\sc Magma}, we determine that $H\circ \mathbb{F}_p^*$ has a regular orbit on $V$,  a contradiction.

Hence $n\leq 11$. First suppose that $W$ is not a basic spin module.
 Lemma \ref{2An reduction} implies that  $V\downarrow N=W$ and $\dim_k(W)$ is listed in Table \ref{tab: double small}. Using \cite{BAtlas,GAP4,SpinSym}, we determine that if $(n,q)$ is one of $(8,9)$, $(10,5)$ or $(11,5)$, then there is no faithful irreducible $kH$-module of dimension 24, 48 or 56 respectively, a contradiction.
Thus $(n,q)$ is one of $(8,7)$, $(9,3)$ or $(10,3)$. If $\varepsilon=+$, then $k=\mathbb{F}_{p^2}$ by \cite{GAP4},  a contradiction, and if $\varepsilon=-$, then using {\sc Magma}, we determine that $H\circ \mathbb{F}_p^*$ has a regular orbit on $V$,  a contradiction. 

Thus $W$ is a basic spin module, in which case  $\dim_{k}(V)=\delta(H)$ by  \cite{BAtlas,GAP4,SpinSym}. Moreover, $(n,q,\varepsilon,\dim_{\mathbb{F}_p}(V))$ is one of 
$(8,3,-,8)$, $(8,5,\pm,8)$, $(9,3,-,8)$, $(10,3,-,16)$,
or else  $H\circ \mathbb{F}_p^*$ has a regular orbit on $V$ by {\sc Magma}. Note that $\dim_{\mathbb{F}_p}(V)=\delta(H)$ since $k=\mathbb{F}_p$. Now (i) holds, so we may assume that $\varepsilon=+$. If $G=2.S_8^+$, then  $G$ has a regular orbit on $V$ by  {\sc Magma}, a contradiction, so $G=2.S_8^+\circ\mathbb{F}_5^*$, as desired.

Conversely, suppose that $\dim_{\mathbb{F}_p}(V)=\delta(H)$ and  either $\varepsilon=-$ and $(n,p)$ is one of $(8,3)$, $(8,5)$, $(9,3)$ or $(10,3)$, or  $\varepsilon=+$ and $(n,p)=(8,5)$ and $G=2.S_n^+\circ\mathbb{F}_p^*$. If 
$(n,p)$ is $(8,3)$ or $(9,3)$,  then $|V|<|G|$, and so $G$ has no regular orbits on $V$.
Otherwise, we use {\sc Magma} to check that no orbit is regular. 
 \end{proof}
 
 Using Proposition \ref{regular 2Sn}, we now consider the double cover of the alternating group.

\begin{prop}
\label{regular 2An}
Let $H:=2.A_n$ where $n\geq 8$, and let $G$ be such that $H\leq G\leq H\circ \mathbb{F}_p^*$  where   $p$ is a prime and $p\leq n$. Let $V$ be a faithful irreducible $\mathbb{F}_pH$-module. Then $G$ has no regular orbits on $V$ if and only if $\dim_{\mathbb{F}_p}(V)=\delta(H)$ and either $p=3$ and $n\in\{8,9,10,11,12\}$, or $p=5$ and $n\in \{9,10\}$.
\end{prop}

\begin{proof}
Suppose that $G$ has no regular orbits on $V$. Let $k:=\End_{\mathbb{F}_pG}(V)$ and $q:=|k|$. As in the proof of Proposition \ref{regular 2Sn}, $q$ is either $p$ or $p^2$ since $\mathbb{F}_{p^2}$ is a splitting field for $H$. For $\varepsilon\in \{+,-\}$, let $V^\varepsilon$ be an irreducible $\mathbb{F}_p(2.S_n^\varepsilon)$-module for which $V\leq V^\varepsilon\downarrow H$, which exists by Lemma \ref{H to G}. Since $r(G)=r(A_n)\leq n/2$ by  \cite[Lemma 6.1]{GurSax2003},  Lemma \ref{general bound} implies that  $$\dim_{\mathbb{F}_p}(V)\leq  r(G)\log_p|G|\leq  \tfrac{n}{2} \log_p(n!\tfrac{p-1}{2})=:h(p,n).$$

Suppose that $n\geq 13$.
Lemma \ref{2An reduction} implies that $V$ is a basic spin module and $(n,q)$ is listed in Table \ref{tab: double big}.  We claim  that
$(n,q)\in P$ where 
$$P:=\{(13,3), (13,13), (14,5), (14,13), (15,5), (16,5), (20,5)\},$$
in which case $H\circ \mathbb{F}_p^*$ has a regular orbit on $V$ by {\sc Magma},  a contradiction. 
 If $(n,q)=(17,11)$, then $128=\delta(H)=\dim_{\mathbb{F}_p}(V)\leq  \lfloor h(p,n)\rfloor=124,$
a contradiction.
 In addition, if $(n,q)$ is one of $(15,11)$, $(15,13)$, $(17,7)$ or $(19,3)$, then $q=p^2$ by \cite{GAP4}, a contradiction.

We may assume that $V^\epsilon$ is  a basic spin module. If $V=V^\varepsilon\downarrow 2.A_n$ for some $\varepsilon\in\{+,-\}$, then $G$ has a regular orbit on $V$ by Proposition \ref{regular 2Sn}, a contradiction. Thus $V^\varepsilon\downarrow 2.A_n=V\oplus Vg$ for every $g\in 2.S_n^\varepsilon\setminus 2.A_n$ and $\varepsilon\in\{+,-\}$.  If  $(n,p)$ is one of $(13,5)$, $(13,7)$, $(13,11)$, $(14,7)$, $(14,11)$, $(16,7)$, $(17,3)$ or $(18,3)$, then  $V^-\downarrow H$ does not split by \cite{GAP4,SpinSym}, a contradiction. Similarly, if $p=3$ and $14\leq n\leq 16$, then $V^+\downarrow H$ does not split by \cite{GAP4,SpinSym}, a contradiction. Lastly, if $(n,p)$ is one of $(13,3)$, $(13,13)$ or  $(15,5)$, then $q=p$ by \cite{GAP4,SpinSym}. Thus $(n,q)\in P$, proving the claim.

Hence $n\leq 12$. First suppose that $V$ is not a basic spin module. Then $(n,q)$ is listed in Table \ref{tab: double small}. If $(n,q)=(8,9)$, then $48=\dim_{\mathbb{F}_p}(V)\leq \lfloor h(p,n)\rfloor=38$, a contradiction. If $(n,q)$ is one of $(8,7)$, $(9,3)$ or $(10,3)$, then $V=V^-\downarrow H$,  so $G$ has a regular orbit on $V$ by Proposition \ref{regular 2Sn}, a contradiction. Lastly, if $(n,q)$ is $(10,5)$ or $(11,5)$, then  we determine that $H\circ \mathbb{F}_p^*$ has a regular orbit on $V$ using {\sc Magma},  a contradiction. 

Thus $V$ is a basic spin module, and we may assume that $V^\varepsilon$ is also a basic spin module. First suppose that $V\neq V^\varepsilon\downarrow H$ for both $\varepsilon\in \{+,-\}$. Using \cite{BAtlas}, we determine that $(n,p)$ is one of $(9,5)$, $(9,7)$, $(10,5)$, $(11,3)$, $(11,5)$ or $(12,3)$. If $(n,p)$ is $(9,7)$ or $(11,5)$, then $H\circ \mathbb{F}_p^*$ has a regular orbit on $V$ by {\sc Magma}, a contradiction. Thus $(n,p)$ is one of $(9,5)$,  $(10,5)$, $(11,3)$ or $(12,3)$, in which case $q=p$ by \cite{BAtlas}, so $\dim_{\mathbb{F}_p}(V)=\delta(H)$. 

Lastly, if $V=V^\varepsilon\downarrow H$ for some $\varepsilon\in \{+,-\}$, then $2.S_n^\varepsilon\circ\mathbb{F}_p^*$ has no regular orbits on $V^\varepsilon$, so $\dim_{\mathbb{F}_p}(V)=\delta(H)$ and $(n,p)$ is one of $(8,3)$, $(8,5)$, $(9,3)$ or $(10,3)$ by Proposition \ref{regular 2Sn}. If $(n,p)=(8,5)$, then  $H\circ \mathbb{F}_p^*$ has a regular orbit on $V$ by {\sc Magma},  a contradiction.

 Conversely, suppose that $\dim_{\mathbb{F}_p}(V)=\delta(H)$ and either $p=3$ and $n\in\{8,9,10,11,12\}$, or $p=5$ and $n\in \{9,10\}$. If $(n,p)$ is one of $(8,3)$, $(9,3)$, $(10,5)$ or $(12,3)$, then $|V|<|G|$, so $G$ has no regular orbits on $V$. Otherwise, no orbit is regular by {\sc Magma}.  
\end{proof}

\begin{proof}[Proof of Theorem $\ref{regular double}$]
Let $d:=\dim_{\mathbb{F}_p}(V)$. If $n\geq 8$, then we are done by Propositions \ref{regular 2Sn} and \ref{regular 2An}, so we may assume that $n\leq 7$. Using \cite{BAtlas,GAP4,SpinSym}, we determine the possibilities for $d$.
  If $(n,p,G,d)$  is not listed in Table \ref{tab: total ex}, then we use {\sc Magma} to prove that $G$ has a regular orbit on $V$. Thus we may assume that $(n,p,G,d)$  is listed in Table \ref{tab: total ex}. If either $(n,p,d)=(7,3,8)$ and $G=2.A_7$, or $(n,p,d)=(5,5,4)$ and $H=2.S_5^+$ or $G=2.S_5^-\circ\mathbb{F}_5^*$ or  $G=2.A_5\circ\mathbb{F}_5^*$, then  no regular orbits exist by {\sc Magma}. 
  Similarly, if  $H= 3.A_6\neq G$ and $(n,p,d)=(6,5,6)$, or if  $H=3.A_7$ and $(n,p,d)=(7,5,6)$ or $(7,7,6)$, then no regular orbits exist by {\sc Magma}.
  Otherwise, $|V|<|G|$, so $G$ has no regular orbits on $V$. 
\end{proof}

 \section{Comments on computations}
 \label{s: comp}

We used functions from \cite{SpinSym} to construct
representations for covering groups of $S_n$ and $A_n$. Various 
representations are also
available via the {\sc Atlas} package \cite{web-atlas}.
{\sc Magma} has an implementation of the Burnside algorithm to
construct all faithful irreducible representations of a finite 
permutation group
over a given finite field.  We used this to construct representations,
either all or those of specified degree, for certain small degree 
permutation groups.
We use our implementation of the algorithm of \cite{GlaLeeOBr2006} to rewrite
a representation over a smaller field.

We used the {\sc Orb} package \cite{ORB} to prove that a 44-dimensional 
representation
of $S_{12}$ over $\mathbb{F}_2$ has a regular orbit and  a 32-dimensional 
representation
of $S_{12}$ over $\mathbb{F}_2$ has no regular orbits.
We used Lemma \ref{strong bound} extensively to decide whether a group 
$G$ has
a regular orbit. Its realisation assumes knowledge of
conjugacy classes of $G$. While these can often be readily
computed, we used the infrastructure of \cite{BaaHolLeeOBr2015} for
these computations with  covering groups for $S_n$ and $A_n$ where $n > 11$.
Most remaining computations reported here are routine and were performed 
using {\sc Magma}.
Records of these are available at \url{http://www.math.auckland.ac.nz/~obrien/regular}.

\bibliographystyle{acm}
\bibliography{jbf_references}

\end{document}